\documentclass[12pt]{amsart}

\usepackage{url,verbatim,hyperref}

\usepackage{amsmath,amstext,amssymb,amsopn,amsthm}
\usepackage{url,verbatim}
\usepackage{mathtools}
\usepackage{enumerate}

\usepackage{color,graphicx}

\usepackage[margin=30mm]{geometry}
\usepackage{eucal,mathrsfs,dsfont}

\allowdisplaybreaks

\newtheorem{theorem}{Theorem}[section]
\newtheorem{corollary}[theorem]{Corollary}
\newtheorem{lemma}[theorem]{Lemma}
\newtheorem{proposition}[theorem]{Proposition}

\theoremstyle{definition}

\newtheorem{remark}[theorem]{Remark}

\theoremstyle{remark}

\numberwithin{equation}{section}

\newcommand{\eps}{\varepsilon}

\newcommand{\calF}{\mathcal{F}}

\newcommand{\calC}{\mathcal{C}}

\newcommand{\calE}{\mathcal{E}}
\newcommand{\calU}{\mathcal{U}}
\newcommand{\calV}{\mathcal{V}}

\newcommand{\E}{\operatorname{\mathds{E}}} 
\renewcommand{\P}{\operatorname{\mathds{P}}} 
\newcommand{\R}{\mathds{R}}

\newcommand{\F}{\mathds{F}}

\newcommand{\prt}{\partial}
\newcommand{\e}{\mathrm{e}}

\newcommand{\wh}{\widehat}
\newcommand{\vphi}{\varphi}
\newcommand{\wt}{\widetilde}

\DeclareMathOperator{\Osc}{Osc}

\def\bone{{\bf 1}}
\def\n{{\bf n}}

\def\bv{{\bf v}}

\def\th{{\theta}}

\def\qq {\qquad}

\title[Reflecting Brownian motion]{Pathwise non-uniqueness for Brownian motion\\ in a quadrant with oblique reflection}
\author{Richard F. Bass and Krzysztof Burdzy}

\address{RFB: Department of Mathematics, University of Connecticut,
Storrs, CT 06269-1009}
\email{r.bass@uconn.edu}

\address{KB: Department of Mathematics, Box 354350, University of Washington, Seattle, WA 98195}
\email{burdzy@uw.edu}

\thanks{Research supported in part by Simons Foundation Grants 506732 and 928958. }

\begin{document}

\begin{abstract}
Consider the Skorokhod equation in the closed first quadrant:
\[ X_t=x_0+ B_t+\int_0^t{\bf v}(X_s)\, dL_s,\]
where $B_t$ is standard 2-dimensional Brownian motion,
$X_t$ takes values in the quadrant for all $t$,  and $L_t$ is a process
that starts at 0, is non-decreasing and continuous, and increases only at those times when $X_t$ is on the boundary of the quadrant.
Suppose ${\bf v}$ equals $(-a_1,1)$ on the positive $x$ axis, 
equals $(1,-a_2)$ on the positive $y$ axis, and ${\bf v}(0)$ points into
the closed first quadrant. Let $\theta_i=\arctan a_i$, $i=1,2$. 
It is known that there exists a solution to the Skorokhod equation for all $t\geq 0$ if and only
if $\theta_1+\theta_2<\pi/2$ and moreover the solution is unique if $|a_1a_2|<1$. 

Suppose now that $\theta_1+\theta_2<\pi/2$, $\theta_2<0$, $\theta_1>-\theta_2>0$
and $|a_1a_2|>1$. We prove that for a large class of $(a_1,a_2)$, namely those for which 
\[\frac{\log|a_1|+\log|a_2|}{a_1+a_2}>\pi/2,\]
pathwise uniqueness for the Skorokhod equation fails to  hold.
\end{abstract}

\maketitle

\section{Introduction}\label{sect-intro}

We consider Brownian motion in the upper right closed quadrant $D$ with 
oblique reflection
on the boundaries.  
Let ${\bf n}(x)$ denote the unit inward pointing normal vector at $x\in \partial D$ and
let ${\bf v}(x)$ denote the vector of reflection, 
normalized so that ${\bf n}(x)\cdot {\bf v}(x)=1$ for each $x\in \partial D$. In this paper we
restrict attention to the case where ${\bf v}$ is equal to a constant on each of the $x$ and $y$ axes.
The value of ${\bf v}(0)$ turns out to be  immaterial as long as it points into
$D$; see Remark \ref{rem-RW}.
Such processes have been extensively studied. They arise, for example, as the limit
of certain queueing models. See \cite{W95} for references.

Many properties  of obliquely reflecting Brownian motion (ORBM) are 
governed by a parameter $\alpha$ which is defined as follows. 
Let $\theta_1, \theta_2$ be the angles of reflection on the $x$ and $y$ axes, 
resp., with positive $\theta$ pointing towards the origin, so that
${\bf v}(z)=(-\tan \theta_1,1)$ at all points $z$ on the $x$ axis except the origin
and ${\bf v}(z)=(1,-\tan \theta_2)$ for all $z$ on the $y$ axis except the origin.
Let 
\begin{equation}\label{def-alpha}
\alpha=\frac{\theta_1+\theta_2}{\pi/2}.
\end{equation}

When $\alpha<1$, ORBM will be a semimartingale (see \cite{W85}) and 
can be represented by what is 
known as the Skorokhod equation.
If $B_t=(B_t^1,B_t^2)$ is standard 2--dimensional Brownian motion
adapted to some filtration $\{\calF_t\}$, 
the Skorokhod equation is
\begin{equation}\label{def-skor}
X_t=x_0+B_t+\int_0^t {\bf v}(X_s)\, dL_s, \qquad t\ge 0.
\end{equation}
Here $L_t$, known as the local time on the boundary, is a non-decreasing continuous
process starting at 0 that increases only at those times when $X_t\in \partial D$.

When does there exist a solution to the Skorokhod equation and when is that
solution pathwise unique? If $\alpha<1$,  a solution to the Skorokhod 
equation exists (\cite{W95}). It is important to note that this 
solution may not be adapted to the filtration generated by $\{B_t\}$. 

Let 
\begin{equation}\label{def-ai}
a_i=\tan \theta_i, \qquad i=1,2.
\end{equation} 
Set 
\begin{equation}\label{def-beta}
\beta=|a_1a_2|.
\end{equation}
It is known (see Section  \ref{sect-prelim}) that if $\beta<1$, then the solution to the Skorokhod equation is pathwise unique and the solution
is adapted to the filtration generated by $\{B_t\}$.

Let
\begin{equation}\label{def-psi}
\psi=\frac{\log |a_1|+\log |a_2|}{(\theta_1+\theta_2)(a_1+a_2)}.
\end{equation}
The main result of this paper is that if $0<\alpha<1$, $\beta>1$, 
$a_1$ and $a_2$ are of opposite signs, and
$\psi>1/\alpha$, then
for almost every $\omega$ there is more than one solution to the Skorokhod equation \eqref{def-skor}. 
Moreover there does not
exist a solution that is adapted to the filtration generated by $\{B_t\}$.
See Section \ref{sect-prelim} for a precise statement.

The reason for the restriction $\alpha<1$ is that ORBM is a semimartingale if and
only if $\alpha<1$, and so the
Skorokhod equation makes sense only in this case. 
When $\alpha\le 0$ the ORBM is transient and upon leaving
the origin never returns to the origin again. Our techniques do not work in this case.
We also do not know what happens in the case  $\beta>1,\psi\le 1/\alpha$.   

There exist processes, in fact
strong Markov processes,
that are uniquely defined in terms of a submartingale problem 
when $\alpha\ge 1$ (see \cite{VW}), but they
are not semimartingales. See \cite{KanRam} for a representation 
of such processes as Dirichlet processes.  

The following historical review  refers to the form of the problem given in \eqref{R-def}-\eqref{def-HR} and later in Section \ref{sect-prelim}.
Mandelbaum \cite{MDCP} showed with
\begin{align*}
R&=\begin{pmatrix} 
        1 & -2  \\
        1& 1 \\
        \end{pmatrix}
\end{align*}
that there is non-uniqueness for the corresponding deterministic Skorokhod equation  (where $B$ is replaced by any continuous driving function $\chi $), even though the R-matrix is a P-matrix, which implies uniqueness for a time-discretized version of the problem. Specifically, Mandelbaum introduced homogeneous differential inequalities such that, for an arbitrary $R$, uniqueness holds for the deterministic Skorokhod equation if and only if zero is the unique solution of the differential inequalities. His proof was constructive in that he gave a procedure for creating a driving path $\chi$ from a non-zero solution $Y$ to the differential inequalities. For the above $2\times 2$ R-matrix, Mandelbaum depicted the corresponding $Y $ in the form of a snail-shell converging to the origin in $\R^2$ ($\chi$ was not shown). Various renderings of the corresponding $\chi$ (which is intricate) and the resulting reflected path, have been given by various people over the years. One published source is the Ph.D. thesis \cite{Whit}, which in a background chapter recounts a form of $\chi$ and the reflected path that emerged from discussions of Mandelbaum with Robert Vanderbei in the mid 1980s (private communication). Subsequent to Mandelbaum, but independently, Bernard and El Kharroubi \cite{BEK} demonstrated that a driving $\chi$, for which non-uniqueness holds, can in fact be linear; their example, however, lives in $\R^3$ and its intricacy stems from the matrix $R$. 

There is a considerable literature devoted to the construction of 
reflected Brownian motion in domains other than quadrants. That research includes theorems on existence and uniqueness of solutions to stochastic differential equations representing reflected Brownian motion. 
Techniques used in that area are considerably different from those in this and related papers. Some references include \cite{BBgam,DI1,DI2,BCMR,LS,VW}.

In Section \ref{sect-prelim} we summarize some results from the literature
that we need, state the main theorem, and give some ideas how the proof goes.
Section \ref{sect-conformal} shows how Brownian motion and certain Bessel processes
transform under certain conformal mappings.
In Section \ref{sect-excursions} we review some excursion theory and extend some of 
those results to ones we will need.
A computation of the expected value of the displacement of $\wh X$, a process
related to our
solution, is given in Section \ref{sect-moving}.
Sections \ref{sect-invariant} and \ref{sect-harnack} discuss the existence
of an invariant measure for a certain Markov chain and a Harnack
inequality for this chain, resp. A strong law of large numbers for the number
of excursions from one side of $D$ to the other is given in Section \ref{sect-slln}.
Section \ref{sect-quadrant} provides the key estimate needed, and then in
Section \ref{sect-nonuni} we use this to prove our non-uniqueness
results. Finally, in Section \ref{sect-notation} we give an index of the notation 
we use.

\subsection{Acknowledgments}
We are indebted  to Ruth Williams for her encouragement, technical advice, literature review,
and general discussions of the problem. Our project would not have been completed
in the current shape without her help. The second author
would like to express his gratitude to the University of California, San Diego,
for hosting him during his sabbatical leave in the spring of 2023. Substantial
progress on the problem was made during the visit. We are grateful to Avi Mandelbaum
for sending us many articles related to the problem, including his seminal 
unpublished paper \cite{MDCP}. 
We thank Z.-Q.~Chen, P.~Fitzsimmons, C.~Landim and R.~Vanderbei for very useful advice.

\section{Preliminaries}\label{sect-prelim}

Many of the previous results concerning uniqueness are stated for
 a slightly different but equivalent formulation of the Skorokhod equation. 

We let $B_t=(B^1_t,B^2_t)$ be a standard Brownian motion started at $x_0\in D$, let $R$ be the matrix
\begin{align}\label{R-def}
R&=\begin{pmatrix} 
        1 & -a_1  \\
        -a_2& 1 \\
        \end{pmatrix},
\end{align}
and let 
\begin{equation}\label{def-HR}
X_t=B_t+RM_t,
\end{equation}
where $M_t$ is a two dimensional process, each of whose coordinates are 
non-decreasing continuous processes started at 0 with $M^1_t$ increasing 
only when $X_t$ is on $\Gamma_d$
and $M^2_t$ increasing only when
$X_t$ is on $\Gamma_u$, 
where
\begin{equation}\label{def-gamma}
\Gamma_u=\{(x,y): x=0, y\ge 0\}, \qquad \Gamma_d=\{(x,y): y=0,x\ge 0\}.
\end{equation}

It is easy to see that if $L_t$ is the local time on the boundary in \eqref{def-skor}, 
then $L_t=M_t^1+M_t^2$ and 
$$M_t^1=\int_0^t 1_{\Gamma_d}(X_s)\, dL_s, \qquad M_t^2=\int_0^t 1_{\Gamma_u}(X_s)\, dL_s$$ give the 
relationship between this formulation and that in \eqref{def-skor}.

\begin{remark}\label{rem-RW}
In view of Reiman and Williams \cite{RW}, neither $M^1$ nor $M^2$ charge
the origin, and therefore  the reflection vector at the origin is of little importance as long as it is a fixed vector pointing into $D$.

A solution to \eqref{def-skor} is clearly a semimartingale, and by \cite{TW}
the law of the solution is equal to the law of the corresponding solution 
to \eqref{def-HR}. In particular $L$ from \eqref{def-skor} does not charge
the origin, and again the reflection vector at the origin can be any fixed
vector pointing into $D$.
\end{remark}

Harrison and Reiman \cite{HR}  used the contraction mapping principle to show there exists a
solution to \eqref{def-HR} which is pathwise unique 
provided $Q=I-R$ is the transition matrix
of a certain class of  transient Markov chains. It was observed in \cite[Theorem 2.1]{W95}
 that
the proof of \cite {HR} carries through provided the spectral radius of $|Q|$
is strictly less than 1, where $|Q|$ is the matrix whose entries are the 
absolute values of the corresponding entries of $Q$.
In this case, as we vary the starting point $x_0$ we can obtain a strong
Markov process.
In fact the result of \cite{HR} and \cite{W95} is valid in higher dimensional
orthants.

Looking only at the 2--dimensional case, an elementary calculation shows that the eigenvalues of $|Q|$ are
$\pm \sqrt{|a_1a_2|}$, and therefore pathwise uniqueness of \eqref{def-HR} holds if
$\beta=|a_1a_2|<1.$

Suppose $a_1,a_2\ge 0$. If 
$\alpha<1$, then 
$\theta_1+\theta_2<\pi/2$, and hence
\begin{align}
\beta&=a_1a_2=(\tan\theta_1)(\tan \theta_2)
< (\tan \theta_1)(\tan(\tfrac{\pi}2-\theta_1))
\label{semimg-beta}\\
&= (\tan \theta_1)(\cot \theta_1)=1.\nonumber
\end{align}
Therefore the question of  pathwise existence and uniqueness for the
Skorokhod equation is completely resolved when $a_1,a_2\ge 0$ and $\alpha<1$.

A $2 \times 2$ matrix $R$ is said to be completely--$S$ if the 
diagonal terms are strictly positive and there exist $z_1, z_2>0$
such that if $Z$ is the $2 \times 1$ column matrix whose entries are $z_1$ and
$z_2$, then the entries of $RZ$ are strictly positive.

For the $R$ given in \eqref{R-def}, this translates to
\begin{equation}\label{eq-comple-S}
z_1-a_2z_2>0, \qquad -a_1z_1+z_2>0.
\end{equation}
When $a_1,a_2>0$, this  is equivalent to
$z_1/z_2>a_2$ and  $z_1/z_2<1/a_1$, and these two equations are solvable
only if $\beta<1$.
If one or both of $a_1, a_2$ are negative, 
then the equations \eqref{eq-comple-S} are solvable and
$R$ is  completely--$S$

Sometimes the definition of completely--$S$ is stated to allow $z_1,z_2\ge 0$ rather than $z_1,z_2>0$. The two definitions are equivalent: if the entries of
$RZ$ are strictly positive for some $z_1,z_2\ge 0$, they will also
be strictly positive for $z_1+\eps, z_2+\eps$ for small enough $\eps$.

An important fact is that if $R$ is completely--$S$, then 
the Skorokhod equation \eqref{def-HR} will always have a solution,
even when $B^1_t$ and $B_t^2$ are replaced by any continuous functions
whatsoever. 
We have the following proposition; see
 \cite{BEK}, \cite{MVdH},  or \cite{DW}
for the proof.

\begin{proposition}\label{prelim-Lcompletely}
 Let $t_0>0$.
Suppose $R$ is a completely-$S$ matrix,
$x_0\in D$, and $f:[0,t_0]\to \R^2$ is continuous with $f(0)=0$.\\
(i) There exists a function $g:[0,t_0]\to \R^2$ that is continuous and is
a solution to 
\begin{equation}\label{prelim-e100.1}
g(t)=x_0+f(t)+ RM(t),
\end{equation}
 where $M(t)$
is a function with values in $\R^2$ each of whose components is non-decreasing 
and continuous and where $M^i$, $i=1,2$,  increases only when $g$ is in 
$\Gamma_d, \Gamma_u$, resp. \\
(ii)
Let $g:[0,t_0]\to \R^2$ be any function that satisfies the conditions of (i).
There exists a constant $\chi_1$ depending only on $a_1$ and $a_2$ but not $t_0, t_1$ or 
$t_2$ such that for all $0\le t_1\le t_2\le t_0$
\begin{equation}\label{osc-bound}
\sup_{u_1,u_2\in [t_1, t_2]}\Big(|g(u_2)-g(u_1)|+|M(u_2)-M(u_1)|\Big)
\le \chi_1  \sup_{u_1,u_2\in [t_1,t_2]}|f(u_2)-f(u_1)|.
\end{equation}
\end{proposition}

Returning to \eqref{def-HR} with $B$ a 2--dimensional Brownian motion, 
Proposition \ref{prelim-Lcompletely}(i) does not  assert that the solution is adapted to the filtration
generated by $\{B_t\}$.
When the pair $(X,B)$ satisfies \eqref{def-HR} and $X$
 is adapted to the filtration generated by $\{B_t\}$, the
pair $(X,B)$ is called a strong solution. 
We
will see later (Section \ref{sect-nonuni}) that pathwise uniqueness is essentially equivalent
to the solution being a strong solution.

Suppose that again $B_t$ is 2--dimensional Brownian motion.
We say a pair  $(X,B)$ solving  \eqref{def-HR} is a weak solution if $B$ is adapted to the
filtration generated by $\{X_t\}$.
For existence and uniqueness of weak solutions we refer the reader to 
Taylor and Williams \cite{TW}.

The techniques we develop in this paper do not say anything about the case 
$\alpha\le 0$, $\beta> 1$. Therefore in the remainder of this paper we consider
only the case where $0<\alpha<1$, $\beta>1$, and one of the $a_i$ is positive, the
other  negative.
To be specific, we take $\theta_1>0$ and $\theta_2<0$.

\begin{figure}[ht]
\includegraphics[width=0.4\linewidth]{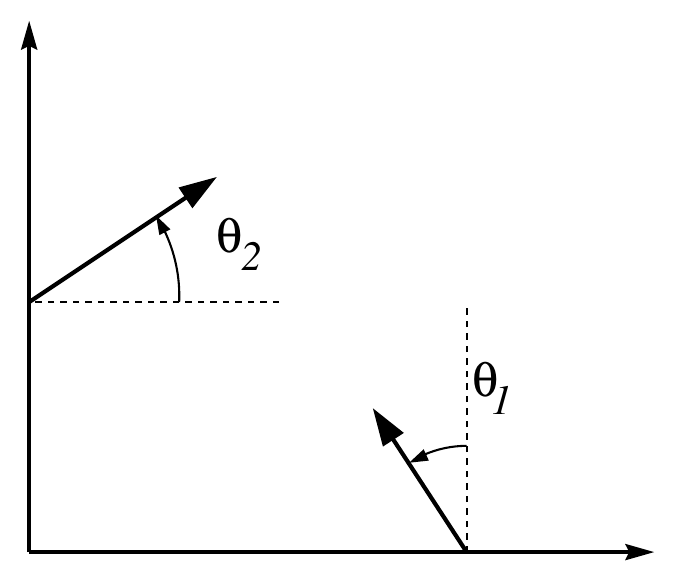}
\caption{According to our conventions, $\th_1>0$ and $\th_2<0$.}
\label{fig9}
\end{figure}

We can now state our main theorem. Recall the definition of $\psi$ 
given in \eqref{def-psi}.

\begin{theorem}\label{main-theorem}
Suppose $a_1>0$, $a_2<0$, $\alpha>0$, $\beta>1$, and $\psi>1/\alpha$. Then there exist two processes $X$ and $Y$
solving \eqref{def-skor} such that with probability one 
$$\sup_{s\ge 0} |X_s-Y_s|>0.$$
Moreover no strong solution to
the Skorokhod equation exists.
\end{theorem}

\begin{remark}
(i)
The conditions $0<\alpha<1$, $\beta>1$, and $\psi>1/\alpha$ are satisfied by some $\theta_1$ and $\theta_2$. For example, take $\theta_1=3\pi/8-\eps$ and $\theta_2=-3\pi/8 +2\eps$ for a very small $\eps>0$.

However, there exist $\theta_1$ and $\theta_2$ satisfying 
$0<\alpha<1$ and $\beta>1$ but not satisfying  $\psi>1/\alpha$. For example, take $\theta_1=7\pi/16$ and $\theta_2=-\pi/8 $.

(ii)
The region  $(\theta_1,\theta_2)\in[\pi/4,\pi/2]\times [-\pi/2,0]$ is depicted in Fig. \ref{fig7}. 
The conditions $0<\alpha<1$ and $\beta>1$ are satisfied in the triangle to the right of 
the red and green lines. The yellow region is where 
the condition $\psi>1/\alpha$ is satisfied. 
The picture was generated using Mathematica.
\end{remark}

\begin{figure}[ht]
\includegraphics[width=0.4\linewidth]{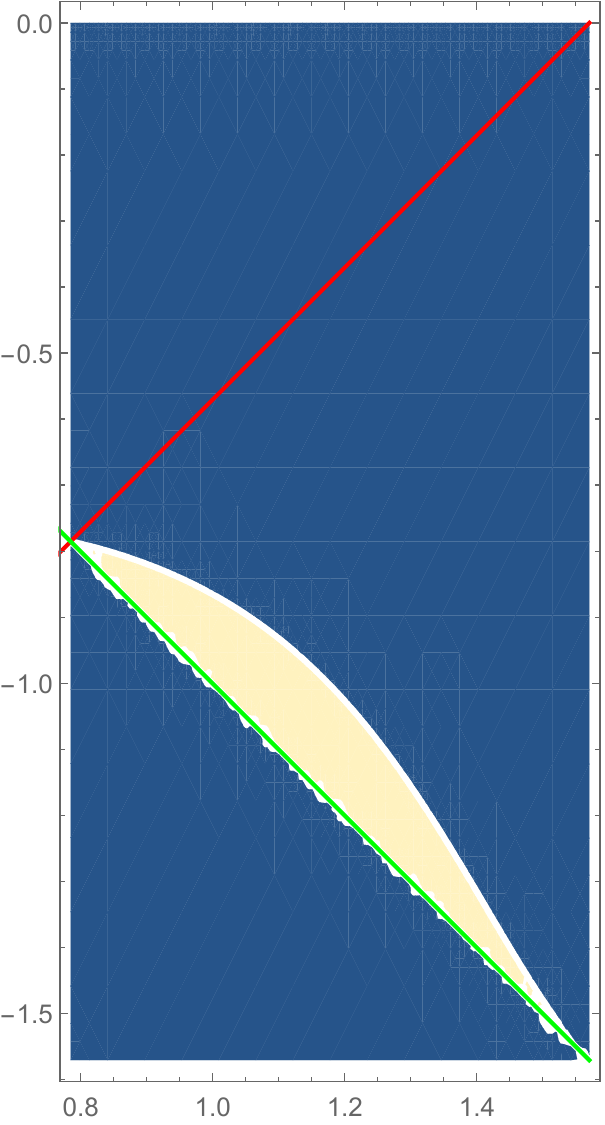}
\caption{ The region  $(\theta_1,\theta_2)\in[\pi/4,\pi/2]\times [-\pi/2,0]$. }
\label{fig7}
\end{figure}

The proof is quite technical 
so we outline some of the steps.

Suppose we have two solutions $X=(X^1,X^2)$ and $Y_n$ to \eqref{def-skor} 
driven by the same Brownian motion $B$, with $X$ started on the $y$ axis
a distance approximately $2^{-n}$ away from the origin
and $Y_n$ started a distance $\eta_n$ to the right of $X$ on the
same horizontal line. As $Y_n$ moves to the $y$ axis, the 
local time on the $y$ axis for $X$ must increase by $\eta_n$ in order for 
$X$ to remain in $D$. Because the vector of reflection is $(1, -a_2)$,
$X$ has an additional displacement of $a_2\eta_n$ in addition to its
motion due to $B^2_t$. 

Once $Y_n$ arrives at the $y$ axis,
we are now in the situation where $X$ and $Y_n$ are both on the $y$ axis, but
a distance $|a_2|\eta_n$ apart. We then watch until the pair $X$ and $Y_n$,
still vertically aligned, move so that the lower of the two is on the
$x$ axis and the upper is $|a_2|\eta_n$ above. By the
argument in the preceding paragraph, after both have reached the
$x$ axis, the distance between $X$ and $Y_n$ is now $|a_1|\,|a_2|\eta_n=\beta\eta_n$,
and the two processes are again on the same horizontal line. 
The net result is that the distance between $X$ and $Y_n$ has increased by
a factor of $\beta$. This is why we need $\beta>1$.

A major part of the  proof consists of estimating the number of times the
distance grows by a factor of $\beta$. We need to ensure that the number
of times is sufficiently large so that when
 $|X|$ is approximately 1, the distance between $X$ and $Y_n$ is 
not negligible, independently of $n$. This is where the condition $\psi>1/\alpha$
appears. In the limit we get two processes started
at 0 which are not identical.

There are some complications. We need to ensure that $X$ does not hit 0,
so we argue that it suffices to look at excursions. It is possible that
$Y_n$ goes only part ways towards the $y$ axis in the argument above, before
moving to the $x$ axis. We also need to bound the probability of this happening.

We could not find a construction of the approximating processes in which the initial distance between them was deterministic because the variance of the distance between them at time 1 would be too large for our purposes. 
To overcome this challenge, we  ``look into the future'' and choose the initial distance so that it makes the two processes have a distance of order 1 at a time of order 1.

Throughout this paper  the letter $c$ with subscripts will denote constants whose exact value is unimportant and which may change from occurrence to occurrence.

\section{Conformal mappings}\label{sect-conformal}

In order to facilitate the computations of the expected number of excursions, we will use two conformal mappings.
The first is
\begin{equation}\label{def-f}
\F(z)= e^{i(\pi/2-\theta_1)} z^\alpha,
\end{equation}
which maps $D$ into a wedge $\wt D=\F(D)$ so that the angles of reflection
are horizontal and point at each other.

Fig. \ref{fig8} shows $\wt D$ with the vectors of reflection horizontal and pointing at each other. The angles in the picture are $\chi =\theta_1+\theta_2$ and $\gamma = \pi/2-\theta_1$.

\begin{figure}[ht]
\includegraphics[width=0.8\linewidth]{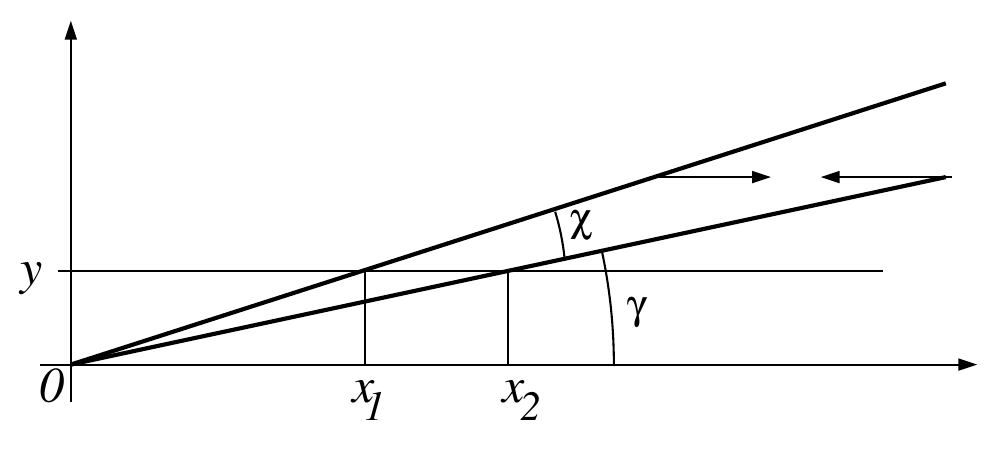}
\caption{  }
\label{fig8}
\end{figure}

The second conformal mapping, 
\begin{equation}\label{def-g}
g(z)=\log z,
\end{equation}
maps the wedge $\wt D$ into a strip $\wh D =\log (\wt D)$.

In this section we determine the law of the time change of the image of $X$ under $\F$, and then the image of a closely related process under $g$.

Since $\F$ maps the vectors of reflection into horizontal vectors, there is
no push for the vertical component of $\F(X_t)$ due to hitting the boundary,
and so we expect the vertical component to be a time change of Brownian motion.

Here is the precise statement. 

\begin{proposition}\label{conf-p1}
Let $X$ be a solution to \eqref{def-skor} stopped on hitting the origin.
Let $A_t=\int_0^t |\F'(X_s)|^2\, ds$ and $C_t =\inf\{s: A_s\ge t\}$. Let
$\wt X_t=\F(X_{C_t})$. Then 
$\wt X$ solves the following Skorokhod equation in $\wt D$:
\begin{equation}\label{def-wtX}
\wt X_t=\F(x_0)+\wt B_t+\int_0^t {\wt{\bf v}}(\wt X_s)\, d\wt L_s,
\end{equation}
where $\wt B_t$ is a standard Brownian motion in $\R^2$, 
${\wt{\bf v}}(z)$ points horizontally into the interior of $\wt D$,
${\wt{\bf v}}(z)$ is normalized so that the component normal to the 
boundary of $\partial \wt D$ is equal to 1,
$\wt L$ is a continuous non-decreasing process starting at 0 that increases
only when $\wt X_t$ is in $\partial \wt D$, and $\wt X$ is stopped
upon hitting the origin.
\end{proposition}

When we work with $\wt X$ we only start the process away from the origin and we 
will subsequently condition it to never hit 0, 
hence we do not need to worry what happens if $\wt X$ were to continue after
hitting the origin.

\begin{proof}
If $\F(z)=U(z)+iV(x)$ (we use upper case $U$ and $V$ to avoid any 
confusion with the vectors of reflection), recall that $U$ and $V$ are
harmonic, hence $U_{xx}+U_{yy}=0$ in the interior of $D$ and the same for
$V$. Here $U_x$ is the partial derivative with respect to $x$ and the notations
$U_y$, $U_{xx}$, $U_{xy}$, $U_{yy}$ are the respective partials. 
If $B_t=(B_t^1,B_t^2)$ is the
planar Brownian motion in \eqref{def-skor}, the quadratic variations are
$d\langle {B^1} \rangle_t=d\langle {B^2} \rangle_t=dt$ and $d\langle {B^1,B^2} \rangle_t=0$.
Let $\calU_t=U(X_t)$, $\calV_t=V(X_t)$.

Using Ito's formula,
\begin{align*}
d\calU_t&=U_x(X_t)\, dB^1_t+U_y(X_t)\, dB^2_t+U_x(X_t) {\bf v}_1(X_t)\, dL_t
 +U_y(X_t){\bf v}_2(X_t)\, dL_t\\
&\qquad + \tfrac12 U_{xx}(X_t)\, d\langle {B^1} \rangle_t 
 +U_{xy}(X_t)\, d\langle {B^1,B^2} \rangle_t+\tfrac12 U_{yy}(X_t)\, d\langle {B^2} \rangle_t\\
& =U_x(X_t)\, dB^1_t+U_y(X_t)\, dB^2_t+U_x(X_t) {\bf v}_1(X_t)\, dL_t
 +U_y(X_t){\bf v}_2(X_t)\, dL_t,
\end{align*}
and similarly for $\calV_t$. 

It is standard that  $(\calU_{C_t},\calV_{C_t})$ is a Brownian motion except for the local time terms (see \cite[pp.\ 310--311]{Ba-pta} for an exposition), so we focus on the 
local time terms. These terms  can be written
as
${\bf q}(\calU_t,\calV_t) \, dM_t,$
where 
\begin{align*}
{\bf q}_0(w_1,w_2)
&= b(z)\Big(U_x(z){\bf v}_1(z)
 +U_y(z){\bf v}_2(z),
 V_x(z){\bf v}_1(z)
+V_y(z){\bf v}_2(z)\Big),\\
z&=\F^{-1}(w_1,w_2),\\
dM_t &=(1/b(\calU_t,\calV_t))\, dL_t.
\end{align*}
Here $b$  is chosen so the normal component of ${\bf q}$
is equal to 1.

We let $\wt L_t=M_{C_t}$ and note that since $C_t$ is strictly increasing and
continuous, then $\wt L_t$ is non-decreasing
and continuous.

It remains to show that ${\bf q}$ has the desired direction. At any point $z\in\partial D$ except for the origin, $\F$ is analytic in a
neighborhood of $z$ and hence conformal. 
If $z\in \Gamma_d$, the angle between the vector $(1,0)$ and ${\bf v}$ is
$\tfrac{\pi}2+\theta_1$. 
It is easy to check that
the angle between $\Gamma_d$ and $\F(\Gamma_d)$ is $\tfrac{\pi}2 -\theta_1$.
Therefore the angle between ${\bf q}(\F(z))$ and
the vector $(1,0)$ is equal to $\pi$, as claimed. The argument for $z\in \Gamma_u$
is similar.
\end{proof}

For our second result, we suppose $\wt X$ is as in Proposition \ref{conf-p1},
except that $\wt X^2$ is now a 3-dimensional Bessel process.
More precisely, we suppose
\begin{equation}\label{conf-e1}
\wt X_t=\wt x_0+\wt B_t+\wt E_t                       
+\int_0^t \wt {\bf v}(\wt X_s)\, d\wt L_s,
\end{equation}
where $\wt B_t$ is a standard Brownian motion in $\R^2$, $\wt{\bf v}$ and
$\wt L$ are as in Proposition \ref{conf-p1}, $\wt x_0=f(x_0)$,
$\wt E^1_t=0$ for all $t$, and $$\wt E^2_t=\int_0^t \frac{1}{\wt X^2_s}\, ds.$$. 

Let  $\wh D=\log(\wt D)$, $\wh \Gamma_u=\log(\wt \Gamma_u)$,
and $\wh \Gamma_d=\log(\wt \Gamma_d)$.

\begin{proposition}\label{conf-p2} 
Let $\wt X_t$ be as in \eqref{conf-e1}. Let $\wt A_t=\int_0^t |1/\wt X_s|^2\, ds$
and $\wt C_t=\inf\{s: \wt A_s\ge t\}$. 
Let $\wh X_t=\log(\wt X_{\wt C_t})$. Then 
$\wh X_t$
solves the following Skorokhod equation in $\wh D$:
\begin{equation}\label{def-whX}
\wh X_t=\log(\wt x_0)+\wh B_t+\wh E_t+\int_0^t \wh {\bf v}(\wh X_s)\, d\wh L_s,
\end{equation}
where $\wh B_t$ is a standard Brownian motion in $\R^2$, 
${\wh{\bf v}}$ is at an angle $\theta_1, \theta_2$ with $\wh \Gamma_d,
\wh \Gamma_u$, resp. and has normal component 1,  
$\wh L$ is a continuous non-decreasing process starting at 0 that increases
only when $\wh X_t$ is in $\partial \wh D$,
and 
\begin{equation}\label{conf-e2}
\wh E_t=\int_0^t (1, \cot(\wh X^2_s))\, ds.
\end{equation}
\end{proposition}

\begin{proof}
The local time terms and martingale terms are handled just as in Proposition \ref{conf-p1},
so we focus only on the drift terms.

Let $z=x+iy=re^{i\theta}$ and let $w=U+iV=\log z$. For the partial derivatives
we have 
$$U_x=\frac{x}{x^2+y^2}, \qq U_y=\frac{y}{x^2+y^2}, \qq V_x=-U_y, \qq V_y=U_x.$$

When we apply Ito's formula as in 
Proposition \ref{conf-p1}, we have what we obtained there plus an additional
drift term: if $(\calU_t,\calV_t)=\log(\wt X_t)$, then
$$d\calU_t=\mbox{ martingale + local time term }+U_y(\wt X_t)\, dt$$
and a similar expression for $d\calV_t$ with the $U$'s replaced by $V$'s.
Therefore
\begin{align*}
d\calU_t&= \mbox{ martingale + local time term }
+\frac{\wt X^2_t}{(\wt X^1_t)^2+(\wt X^2_t)^2}\cdot\frac{1}{\wt X^2_t} dt,\\
d\calV_t&=
\mbox{ martingale + local time term }
+\frac{\wt X^1_t}{(\wt X^1_t)^2+(\wt X^2_t)^2}\cdot\frac{1}{\wt X^2_t} dt.\\
\end{align*}

Since $(\calU_t,\calV_t)=\log {\wt X}_t$, then ${\wt X}_t=(e^{\calU_t} \cos \calV_t,
e^{\calU_t} \sin \calV_t)$. We therefore
have
\begin{align*}
d \calU_t&=\mbox{ martingale + local time term }+\frac{1}{e^{2\calU_t}}\, dt,\\
d \calV_t&=\mbox{ martingale + local time term }+
\frac{\cot \calV_t}{e^{2\calU_t}}\ dt.
\end{align*}

Since $\left|{\wt X}_t\right|^{-2}=e^{-2\calU_t}$, a time change argument
using  $\wt A_t=\int_0^t |1/\wt X_s|^2\, ds$ shows
that the drift term is given by \eqref{conf-e2}.
\end{proof}

\section{Some excursion theory}\label{sect-excursions}

The following review of excursion theory is taken from \cite{BCJ} because a crucial element of our argument, just as in \cite{BCJ}, is the proper normalization of several quantities.
See,
e.g., \cite{Mais} for the foundations of excursion theory in the abstract
setting and \cite{Bur} for the special case of excursions of Brownian
motion. Although \cite{Bur} does not discuss reflecting Brownian motion,
all the results we need from that book readily apply in the present
context. 

An ``exit system'' for excursions of  reflecting Brownian motion
$X$ from $\prt D$ is a pair $(L^*_t, H^x)$ consisting of a
positive continuous additive functional $L^*_t$ and a family of
``excursion laws'' $\{H^x\}_{x\in\prt D}$. We will explain below that
$L^*_t = L^X_t$. Let $\Delta$ denote the ``cemetery'' point
outside $\R^2$ and let ${\calC}$ be the space of all functions
$f:[0,\infty) \to \R^2\cup\{\Delta\}$ which are continuous and
take values in $\R^2$ on some interval $[0,\zeta)$, and are equal
to $\Delta$ on $[\zeta,\infty)$. For $x\in \prt D$, the excursion
law $H^x$ is a $\sigma$-finite (positive) measure on $\calC$,
such that the canonical process is strong Markov on
$(t_0,\infty)$, for every $t_0>0$, with the transition
probabilities of Brownian motion killed upon hitting $\prt D$.
Moreover, $H^x$ gives zero mass to paths which do not start from
$x$. We will be concerned only with ``standard'' excursion
laws; see Definition 3.2 of \cite{Bur}. For every $x\in \prt D$ there
exists a unique standard excursion law $H^x$ in $D$, up to a
multiplicative constant.

Excursions of $X$ from $\prt D$ will be denoted $\e$ or $\e_s$.
If $s< u$, $X_s,X_u\in\prt D$, and $X_t \notin \prt D$ for
$t\in(s,u)$ then $\e_s = \{\e_s(t) = X_{t+s} ,\,  t\in[0,u-s)\}$ and
$\zeta(\e_s) = u -s$. By convention, $\e_s(t) = \Delta$ for $t\geq
\zeta$, so $\e_t \equiv \Delta$ if $\inf\{s> t: X_s \in \prt D\} =
t$. Let ${\calE}_u = \{\e_s: s \leq u\}$.

Let $\sigma_t = \inf\{s\geq 0: L^*_s \geq t\}$
and let $I$ be the set of left endpoints of all connected
components of $(0, \infty)\setminus \{t\geq 0: X_t\in \partial
D\}$.
 The following is a
special case of the exit system formula of \cite{Mais},
\begin{align}\label{s25.11}
\E \left[ \sum_{t\in I} V_t \cdot f ( \e_t) \right]
= \E \int_0^\infty V_{\sigma_s}
 H^{X(\sigma_s)}(f)\, ds = \E \int_0^\infty V_t H^{X_t}(f)\, dL^*_t, 
\end{align}
where $V_t$ is a predictable process, either non-negative or such that 
at least one of the terms of \eqref{s25.11} with $V$ replaced by $|V|$ is finite,  and $f:\, {\calC}\to[0,\infty)$ is a universally measurable function which
vanishes on excursions $\e_t$ identically equal to $\Delta$. Here
and elsewhere $H^x(f) = \int_{\calC} f\, dH^x$.

The normalization of the exit system is somewhat arbitrary. For
example, if $(L^*_t, H^x)$ is an exit system and $c\in(0,\infty)$
is a constant then $(cL^*_t, (1/c)H^x)$ is also an exit system.
One can even make $c$ dependent on $x\in\prt D$. Let $\P^y_D$
denote the distribution of Brownian motion starting from $y$ and
killed upon exiting $D$. Theorem 7.2 of \cite{Bur} shows how to choose a
``canonical'' exit system; that theorem is stated for the usual
planar Brownian motion but it is easy to check that both the
statement and the proof apply to reflecting Brownian motion.
According to that result, we can take $L^*_t$ to be the continuous
additive functional whose Revuz measure is a constant multiple of
the arc length measure on $\prt D$ and the $H^x$'s to be standard
excursion laws normalized so that
\begin{align}\label{s24.2}
H^x (A) = \lim_{\delta\downarrow 0} \frac{1}{\delta} \P_D^{x +
 \delta\n(x)} (A),
\end{align}
for any event $A$ in the $\sigma$-field generated by the process on
an interval $[t_0,\infty)$, for any $t_0>0$.

The Revuz measure of $L^X$ is the measure $dx/(2|D|)$ on $\prt D$:
 if the initial distribution of $X$ is the uniform
probability measure $\mu$ in $D$ then 
\begin{align}\label{s25.10}
\E^\mu \int_0^1 \bone_A
(X_s)\, dL^X_s = \int_A\, dx/(2|D|)
\end{align}
for any Borel set $A\subset \prt
D$; see  \cite[Ex. 5.2.2]{FOT}.

The crucial fact for us, proved in \cite{BCJ}, is that $L^*_t=L^X_t$.
 That is,  the normalization of
the local time $L^X_t$ contained implicitly in \eqref{def-skor} and the
normalization of excursion laws $H^x$ given in \eqref{s24.2} match so that
$(dL^X_t, H^x)$ is an exit system for $X_t$ from $\prt D$.

All of the above applies to  the processes $\wt X$ in $\wt D$ and $\wh X$ in $\wh D$. 
We will put tildes   or hats 
 above all the corresponding objects.
 
\section{Moving along the strip}\label{sect-moving}

Define
\begin{align}
S^u_0 &=0,\qquad S^d_0=0, \label{def-S}\\
S^d_k &= \inf\{t\geq S^u_{k-1}: X_t\in \Gamma_d\}, \qq k\geq 1,\notag\\
S^u_k &= \inf\{t\geq S^d_{k}: X_t\in \Gamma_u\}, \qq k\geq 1.\notag
\end{align}
These are the successive hits of $X$ to $\Gamma_u$ and $\Gamma_d$.
Define $\wt S^u_k, \wt S^d_k$ and  $\wh S^u_k, \wh S^d_k$ analogously, but with
$X, \Gamma_u, \Gamma_d$ replaced by $\wt X, \wt \Gamma_u, \wt \Gamma_d$ and
 $\wh X, \wh \Gamma_u, \wh \Gamma_d$, resp.

Recall that $\wh D$ is a horizontal strip with upper boundary $\wh \Gamma_u$
and lower boundary $\wh \Gamma_d$. In this section we compute the expected
value of the difference in $\wh X^1$ between 
$\wh S^d_{k+1}$ and $\wh S^d_k$.

\begin{remark}\label{o21.1}
We begin by recalling several formulas from the theory of one-dimen\-sion\-al diffusions;
see, e.g., \cite[Ch. 15, Sec. 3]{KT}. 
If the generator is given by
$$\frac 12 \sigma^2(x) \frac {d^2} {dx^2}+\mu(x) \frac d {dx},$$ the
scale function $S$ is defined up to an additive constant by
$$S'(x)= \exp\left( - \int^x (2 \mu(y)/\sigma^2(y))\, dy\right)$$
and the speed measure has a density given by
$$m(x)  = \frac 1 {\sigma^2(x) S'(x)}.$$
If $T_c$ denotes the hitting time of $c\in \R$ then for $a<x<b$,
\begin{align}\notag
p(x)&=\P_x(T_b < T_a) = \frac{S(x) - S(a)}{S(b) - S(a)} ,\\
\E_x (T_a\land T_b) &= 
2 p(x) \int_x^b (S(b) - S(y))m(y)\, dy \label{o22.1}\\
&\qquad + 2 (1-p(x)) \int_a^x (S(y)-S(a)) m(y)\, dy.\notag
\end{align}

\end{remark}

Let $\wh \tau_u$ and $\wh \tau_d$ denote the hitting times of $\wh \Gamma_u$ and $\wh \Gamma_d$, resp.
Let $\{\wh H^\cdot\}$ denote the excursion laws corresponding to the process
$\wh X$.  

\begin{lemma}\label{s25.8}
We have
\begin{align}\label{o21.4}
\wh H^{(x,y)}(\wh \tau_u < \infty) &= 
\frac 1 {\cos^2 \theta_1(\tan\theta_1 + \tan\theta_2)}, \qquad (x,y)\in \wh \Gamma_d,\\
\wh H^{(x,y)}(\wh \tau_d < \infty) &= 
\frac 1 {\cos^2 \theta_2(\tan\theta_1 + \tan\theta_2)},
\qquad (x,y)\in \wh \Gamma_u.\label{o21.5}
\end{align}
\end{lemma}

\begin{proof}
When $\wh X$ is in the interior of the strip $\wh D$, $\wh X^2$ is a one-dimensional diffusion with drift coefficient $\cot y$ when $\wh X^2_t=y$
by
Proposition \ref{conf-p2}. We apply Remark \ref{o21.1} to $\wh X^2$ with  $a = \pi/2 - \theta_1$, $b= \pi/2+\theta_2$, $\sigma \equiv 1$, and $\mu(y) = \cot y$.
  Note that
\begin{align}\label{o21.2}
\sin( a) &= \sin( \pi/2 -  \theta_1) = \cos(\theta_1),\qquad
\sin(b)  = \sin(\pi/2 +  \theta_2) =  \cos(\theta_2),\\
\cot(a) &= \cot( \pi/2 -  \theta_1) = \tan(\theta_1),\qquad
\cot(b)  = \cot(\pi/2 +  \theta_2) =  -\tan(\theta_2).\label{o21.3}
\end{align}
By Remark \ref{o21.1},
\begin{equation}
S'(x)
= \exp\left( - \int^x 2 \cot y \, dy\right)
 = 1/\sin^2 x,
\end{equation}
hence
$$S(x)    = \int^x \sin^{-2} z\, dz
= -\cot x$$
and
$$p(x)=\frac{S(x) - S(a)}{S(b) - S(a)}
=\frac{-\cot x + \cot a}
{-\cot b + \cot a}.$$
Consider $x= a+ \delta$ where $\delta>0$ is small. Then 
$$-\cot x + \cot a
= -\delta  \frac d {dy}(\cot a) \Big|_{y=a}+o(\delta)
=\delta/ \sin^2 a+o(\delta).$$
Hence
$$p(x)
=p(a+\delta)=\frac{-\cot x + \cot a}
{-\cot b + \cot a}
= \frac{\delta}
{\sin^2 a(\cot a - \cot b)}+o(\delta).$$
We apply \eqref{s24.2} and use \eqref{o21.2}-\eqref{o21.3} to see that
\begin{align*}
\wh H^{(x,y)}(\wh \tau_u < \infty)
=\lim_{\delta\downarrow 0}
\frac 1 \delta p(a+\delta)
=
\frac{1}
{\sin^2 a(\cot a - \cot b)}
=\frac 1 {\cos^2 \theta_1(\tan\theta_1 + \tan\theta_2)}.
\end{align*}
This completes the proof of \eqref{o21.4}. The proof of \eqref{o21.5} is analogous.
\end{proof}

Set
\begin{equation}\label{def-kappa}
\kappa = \frac1{(\theta_1+\theta_2)(\tan \theta_1+\tan \theta_2)}.
\end{equation}

\begin{proposition}\label{o21.6}
For $k\geq 1$,
\begin{align*}
 \wh \E^{\wh h}_\cdot \left(\wh X^1\left(\wh S^d_{k+1}\right) 
- \wh X^1\left(\wh S^d_k\right)\right)
=1/\kappa.
\end{align*}
\end{proposition}

\begin{proof}
\emph{Step 1}.
According to the exit system formula \eqref{s25.11}, the distribution of the amount of local time $\wh L\left(\wh S^u_k\right)-\wh L\left(\wh S^d_k\right)$ 
accumulated on $\wh \Gamma_d$ between a time when $\wh X$ hits $\wh \Gamma_d$ 
and the first time after that when it hits $\wh \Gamma_u$ is exponential with 
mean equal to $\big(\wh H^{(x,y)}(\wh \tau_u < \infty)\big)^{-1}$. In view of \eqref{o21.4},
\begin{align*}
 \wh \E^{\wh h}_\cdot \left(
\wh L\left(\wh S^u_k\right)-\wh L\left(\wh S^d_k\right)
\right) = \cos^2 \theta_1(\tan\theta_1 + \tan\theta_2).
\end{align*}

The  reflection vector on $\wh \Gamma_d$ is $\wh\bv=(\wh \bv^1,\wh\bv^2)=(-\tan\theta_1,1)$ according to our sign convention, so the last formula implies that
\begin{align*}
 \wh \E^{\wh h}_\cdot \left(
\int_{\wh S^d_k}^{\wh S^u_k}
\wh\bv^1(\wh X_t)\, d\wh L_t\right) = -\tan\theta_1\cos^2 \theta_1(\tan\theta_1 + \tan\theta_2).
\end{align*}
Similarly,
\begin{align*}
 \wh \E^{\wh h}_\cdot \left(
\int_{\wh S^u_k}^{\wh S^d_{k+1}}
\wh\bv^1(\wh X_t)\, d\wh L_t\right) = -\tan\theta_2\cos^2 \theta_2(\tan\theta_1 + \tan\theta_2),
\end{align*}
and so
\begin{align*}
 \wh \E^{\wh h}_\cdot \left(
\int_{\wh S^d_k}^{\wh S^d_{k+1}}
\wh\bv^1(\wh X_t)\, d\wh L_t\right) = -(\tan\theta_1\cos^2 \theta_1+\tan\theta_2\cos^2 \theta_2)(\tan\theta_1 + \tan\theta_2).
\end{align*}

\medskip
\noindent
\emph{Step 2}.
In this step, we will find $
 \wh \E^{\wh h}_\cdot \left(\wh S^d_{k+1}
- \wh S^d_k\right)$.

Formula \eqref{o22.1} is concerned with the expected amount of time that it takes for the diffusion to go from an interior point in the interval $(a,b)$ to one of the boundary points. 
We would like to apply this formula to the reflected process $\wh X^2$. To achieve that, we will ``unfold'' the reflection. More precisely, we will consider a diffusion on the interval $(\pi/2-2\theta_1-\theta_2,\pi/2+\theta_2) $ with the 
same parameters $\sigma\equiv 1$ and $\mu(y)=1/y$ for $y$ in
$(\pi/2-\theta_1,\pi/2+\theta_2) $, 
 the upper half of the interval. 
For $y$ in 
 $(\pi/2-2\theta_1-\theta_2,\pi/2-\theta_1) $, 
the lower half of the interval, 
we take $\sigma\equiv 1$ and make the drift $\mu$ symmetric, 
that is,  $\mu(y) = - \mu (\pi/2-\theta_1 +(\pi/2-\theta_1 -y))$. 
The ``unfolded'' diffusion has no reflection at $\pi/2-\theta_1$. 
It is clear that the expected time for the new diffusion starting from $\pi/2-\theta_1$ to exit $(\pi/2-2\theta_1-\theta_2,\pi/2+\theta_2) $ is the same as for $\wh X^2$ 
 starting from $\pi/2-\theta_1$ to exit $(\pi/2-\theta_1,\pi/2+\theta_2) $.
We will write $\E^*$ to denote the expectation corresponding to the unfolded diffusion. 

In the following calculation we take $p(x)=1/2$ because 
by symmetry it represents the probability that the unfolded diffusion will exit the interval at the upper end. 
By \eqref{o22.1},
\begin{align*}
& \wh \E^{\wh h}_\cdot \left(\wh S^u_{k}
- \wh S^d_k\right)
=\E^*_{\pi/2-\theta_1} (T_{\pi/2+\theta_2}\land T_{\pi/2-2\theta_1-\theta_2}) \\
&= 
2 p(x) \int_x^b (S(b) - S(y))m(y)\, dy
 + 2 (1-p(x)) \int_a^x (S(y)-S(a)) m(y)\, dy\\
&=  \int_x^b (S(b) - S(y))m(y)\, dy
 + \int_a^x (S(y)-S(a)) m(y)\, dy,\\
&= 2 \int_x^b (S(b) - S(y))m(y)\, dy.
\end{align*}
The last equality holds by symmetry.
Hence,
\begin{align*}
& \wh \E^{\wh h}_\cdot \left(\wh S^u_{k}
- \wh S^d_k\right)
=2\int_x^b (S(b) - S(y))m(y)\, dy
=  2 \int_{\pi/2-\theta_1}^{\pi/2+\theta_2} (\tan \theta_2 + \cot y)\sin^2 y\, dy\\
&= (\theta_1 +\theta_2 ) \tan (\theta_2 )+ \sin ^2( \theta_1 )+\sin ( \theta_1 )\cos \theta_1\tan\theta_2 .
\end{align*}
A  calculation, based on  similar unfolding arguments, yields
\begin{align*}
& \wh \E^{\wh h}_\cdot \left(\wh S^d_{k+1}
- \wh S^u_k\right)
=\E^*_{\pi/2+\theta_2} (T_{\pi/2-\theta_1}\land T_{\pi/2+2\theta_2+\theta_1})\\
&=  (\theta_1 +\theta_2 ) \tan (\theta_1 )+
  \sin^2( \theta_2)+ \tan  (\theta_1) \sin( \theta_2)\cos\theta_2 .
\end{align*}
Adding the last two formulas, we obtain
\begin{align*}
 \wh \E^{\wh h}_\cdot \left(\wh S^d_{k+1}
- \wh S^d_k\right)&= (\theta_1 +\theta_2 ) \tan (\theta_2 )+ \sin ^2( \theta_1 )+\sin ( \theta_1 )\cos \theta_1\tan\theta_2 \\
&\qquad + (\theta_1 +\theta_2 ) \tan (\theta_1 )+
  \sin^2( \theta_2)+ \tan  (\theta_1) \sin( \theta_2)\cos\theta_2.
\end{align*}

\medskip
\noindent
\emph{Step 3}.
The process $\wh X$ satisfies the stochastic differential equation 
\begin{align}\label{o21.7}
\wh X_t &= x_0 +\wh B_t + \int_0^t \left(1,\cot\left( \wh X^2_s\right)\right) \,ds+ \int_0^t  \wh \bv(\wh X_s)\, d\wh L_s,
 \qquad \hbox{for } t\geq 0,
\end{align}
where $\wh B$ is two-dimensional Brownian motion.
Hence, the formulas derived in Steps 1 and 2 yield
\begin{align*}
& \wh \E^{\wh h}_\cdot \left(\wh X^1\left(\wh S^d_{k+1}\right) 
- \wh X^1\left(\wh S^d_k\right)\right)
=  \wh \E^{\wh h}_\cdot \left(
\int_{\wh S^d_k}^{\wh S^d_{k+1}}
\wh\bv_1(\wh X_t)\, d\wh L_t\right) 
+  \wh \E^{\wh h}_\cdot \left(\wh S^d_{k+1}
- \wh S^d_k\right)\\
&= -(\tan\theta_1\cos^2 \theta_1+\tan\theta_2\cos^2 \theta_2)(\tan\theta_1 + \tan\theta_2)\\
&\qquad +(\theta_1 +\theta_2 ) \tan (\theta_2 )+ \sin ^2( \theta_1 )+\sin ( \theta_1 )\cos \theta_1\tan\theta_2 \\
&\qquad+ (\theta_1 +\theta_2 ) \tan (\theta_1 )+
  \sin^2( \theta_2)+ \tan  (\theta_1) \sin( \theta_2)\cos\theta_2\\
&=(\theta_1 +\theta_2 )( \tan \theta_1+\tan \theta_2 ).
\end{align*}
\end{proof}

\section{Invariant measures}\label{sect-invariant}

This section is devoted to establishing the existence of an invariant measure, 
or equivalently, a stationary probability distribution, for a Markov chain
defined in terms of $\wt X$.

Define
\begin{equation}\label{def-H}
\wt H_b=\{(x,y)\in \wt D: y=b\}, \qquad H_b=\F^{-1}(\wt H_b), \qquad
 \wh H_b=\log(\wt H_b),
\end{equation}
where $\F$ is the function given in \eqref{def-f}. 
Define
\begin{equation}\label{def-Tb}
T_b=\inf\{t: X_t\in H_b\}
\end{equation} 
and define $\wt T_b$ and $\wh T_b$ analogously,
but with $H_b$ replaced by $\wt H_b$ and $\wh H_b$ and 
$X$ replaced by $\wt X$ and $\wh X$, 
resp.

We start with $X_t$,  an ORBM  in $D$ satisfying \eqref{def-skor}. We define $\wt X$ by means of 
Proposition \ref{conf-p1}. Since $\wt X^2$ is a one-dimensional Brownian motion,
the function
\begin{equation}\label{def-h2}
\wt h(x,y)=y, \qq (x,y)\in \wt D
\end{equation}
is harmonic. We then condition $\wt X$ by using the 
$h$-path transform of Doob with the function \eqref{def-h2}. This gives
rise to  new probability measures 
defined by
\begin{equation}\label{def-hpath}
\wt P^{\wt h}_z(A)=\wh \E_z[\wt h(\wt X_{S}); A]/\wt h(z),
\end{equation}
where $\wt X$ is adapted to a filtration $\{\mathcal{F}_t\}$, $S$ is a stopping
time with respect to $\{\mathcal{F}_t\}$, and $A\in \mathcal{F}_S$.
Note $\partial \wt h(x,y)/\partial x=0$ and $\partial \wt h(x,y)/\partial y=1$.
Therefore
under $\wt P^{\wt h}_z$ the process $\wt X$ is a 3-dimensional Bessel process
in the second coordinate; see \cite[pp.\ 61--62]{Ba-pta}. 
Finally we apply Proposition \ref{conf-p2} to obtain $\wh X$.

Recall that if we have a strong Markov process $Z_t$, a point $z$ is regular for
a Borel set $A$ if $\P^x(\tau_A=0)=1$, where $\tau_A=\inf\{t>0:  Z_t\in A\}$.

We establish this for  $\wt X$, for points $z\in \wt H_1$, 
and 
\begin{equation}\label{def-wtA}
\wt A=\{z=(x,y): y>1\}.
\end{equation}

\begin{lemma}\label{lem-regular}
Let $(\wt X, \P^{\wt h}_z)$ be the strong Markov process given by \eqref{def-hpath}.
Let $\wt A$ be the set defined by \eqref{def-wtA}. Then every point of $\wt H_1$ is regular
for $\wt A$.
\end{lemma}

\begin{proof}

The process $\wt X^2$ is a 3 dimensional Bessel process defined in 
\eqref{conf-e1}, and so solves the stochastic differential equation
$$X_t^2=X_0^2+ B_t^2+\int_0^t \frac{1}{\wt X^2_s}\, ds,$$
where $B_t^2$ is a Brownian motion independent of $\wt X^1_t$. Therefore
$X^2_t\ge X_0^2+B_t^2$. It is well known that for
one-dimensional Brownian motion the origin is regular with
respect to the interval $(0,\infty)$, and using translation invariance 
it follows that for $\wt X^2_t$ the
point $1$ is regular for $(1,\infty)$. Consequently for $\wt X_t$ the point
$(x,1)$ is regular for $\wt A$.
\end{proof}
  
We now define a Markov chain on $\wt H_1$ by using the kernel
\begin{equation}\label{def-Q}
\wt Q(x,dy)=\wt P^{\wt h}_{e^{-1}x}(\wt X(\wt T_1)\in dy), \qquad x,y\in \wt H_1.
\end{equation}

Our main result in this section is that there exists an invariant 
measure for this chain. 

\begin{proposition}\label{invar-p1} There exists an invariant probability $\wt \nu$
for the chain with kernel $\wt Q$. 
\end{proposition}

\begin{proof}
By page 316 of \cite{TW} and  \cite[Corollary 4.6]{SV}, the law of
$\wt P^{\wt h}_{e^{-1}x_n}$ converges weakly to the law of $\wt P^{\wt h}_{e^{-1}x}$ whenever $x_n,x\in D$ and $x_n\to x$. By the 
Skorokhod representation theorem 
(see, e.g.,  \cite[Chapter 31]{Ba-stoch}),  
we can find a probability space $(\Omega', \mathcal{F}',\P')$ and processes $X'_n$ and $X'$ such that the law of 
$X'_n$ under $\P'$ is the same as that of $\wt X$ under $\wt P^{\wt h}_{e^{-1}x_n}$,
the law of $X'$ under $\P'$ is the same as that of $\wt X$ under $\wt P^{\wt h}_{e^{-1}x}$, and $X'_n(t)$ converges uniformly on compacts to $X'(t)$
almost surely  with respect to $\P'$. 

Let $T'_n=\inf\{t: X'_n(t)\in \wt H_1\}$ and define $T'$ similarly.
If $F$ is a continuous function on $\wt H_1$, if $x_n,x\in \wt H_1$, and
if $x_n\to x$, it may not be true that $F(X'_n(T'_n))$ converges to $F(X'(T'))$
for every $\omega$. But the only $\omega$'s for which it is not true are
those for which the path of $X'$ just hits $\wt H_1$ from below but does not
subsequently immediately enter the set $\wt A$ defined in \eqref{def-wtA}.
By the strong Markov property and Lemma \ref{lem-regular}, the probability of such paths is 0.

Therefore, using that $F$ is bounded because $\wt H_1$ is compact,
$$\E' F(X'_n(T'_n))\to \E'F(X'(T')),$$
or 
$$\wt \E^{\wt h}_{e^{-1}x_n} F(X_{T_1})\to \wt \E^{\wt h}_{e^{-1}x} F(X_{T_1}).$$
This is the same as saying $$\wt QF(x_n)=\int F(y) \, \wt Q(x_n,dy)\to \int F(y)\, 
\wt Q(x,dy)=\wt QF(x).$$
Therefore $\wt Q$ maps continuous functions to continuous functions.
 
Since $\wt H_1$ is compact, $\{\wt Q^n(x,\cdot)\}$ is tight for each $x$,
where $\wt Q^n$ is the $n^{th}$ iterate of $\wt Q$.
The existence of an invariant measure  now follows by the Krylov-Bogolyubov theorem 
(\cite[Theorem 1.10]{Hair}).
\end{proof}

As a consequence to Proposition \ref{s25.4} we will show that the
invariant measure is unique and that $\wt Q^n(x,dy)$ converges to $\wt \nu(dy)$ for
every $x$.

\begin{remark}\label{rem-scaling}
If $W_t$ is a one-dimensional Brownian motion or a 3-dimen\-sion\-al Bessel process
started at $w_0$ and $b>0$, then scaling tells us that $\sqrt b W_{t/b}$ is
again a one-dimensional Brownian motion or a 3-dimensional Bessel process, but now
started at $ \sqrt b w_0$. If $\wt X$ is a solution to \eqref{def-wtX}, then
$\sqrt b \wt L_{t/b}$ is again a non-decreasing continuous process that 
increases only when $\wt X_{t/b}$ is in $\partial \wt D$.
Therefore
$\sqrt b\wt X_{t/b}$ is a solution to \eqref{def-wtX}, started at $x_0/\sqrt b$.
We have uniqueness in law for \eqref{def-wtX} by \cite{TW}, hence except for the
starting point, $\sqrt b \wt X_{t/b}$ has the same law as $\wt X_t$. We refer
to this property as scaling.
\end{remark}

Let $\wt \nu(1)$ be the probability distribution on $\wt H_1$ 
given in Proposition \ref{invar-p1}. Define a probability $\wt \nu(e^{m})$
on $\wt H_{e^{m}}$ for integer $m$ by
\begin{equation}\label{def-wtnu}
\wt \nu(e^{m})(A)=\wt \nu(1)(e^{-m}A)
\end{equation}
for Borel subsets $A\subset \wt H_{e^{m}}$,
where $bA=\{by: y\in A\}$.

\begin{remark}\label{rem-invar2}
By scaling and the fact that $\wt \nu(1)$ is a stationary distribution
for the kernel $\wt Q$, we see that if $n\le m$, then the distribution of 
$\wt X(\wt T (e^{m}))$ under $\wt P^{\wt h}_{\wt \nu(e^{n})}$ is 
$\wt \nu(e^{m})$.
\end{remark}

\section{Harnack inequality}\label{sect-harnack}

In this section we prove a Harnack inequality by means of what
is essentially a coupling argument.

\begin{proposition}\label{s25.4}
There exists $q\in(0,1)$ such that for   all integers $i$ and $n$ with $ i \geq n+2 $, 
$x,y\in \wt H_{e^{n}}$, and  $A$ a Borel subset of $\wt H_{e^{i}}$,
\begin{equation}\label{harn-e1.1}
1- q^{i-n-1}
\leq 
\frac
{\wt \P^{\wt h}_x(\wt X(\wt T _{e^{ i}})\in A)}
{\wt \P^{\wt h}_y(\wt X(\wt T _{e^{ i}})\in A)}
\leq \left(1- q^{i-n-1}\right)^{-1}.
\end{equation}
\end{proposition}

\begin{proof}
By scaling it suffices to take $i=0$.  This implies that $n\le -2$.
Define
$$\Phi(z)=\wt P^{\wt h}_z(\wt X(\wt T_1)\in A), \qquad z\in \wt D\setminus\{0\},$$
where $\wt X$ satisfies \eqref{def-wtX} with starting point $z$ and
$\wt P^{\wt h}_z$ is the law of $\wt X$ with starting point $z$. Let $b\in (0,1)$ and $x\in \wt H_{e^{n}}$. We first show there
exists a continuous curve $\gamma$ in $\wt D$ connecting $x$ and $\wt H_1$
that avoids a neighborhood of 0 and such that 
$\Phi(z)\ge b \Phi(x)$
for all $z$ lying on $\gamma$.

By \cite{TW} the solution to \eqref{def-skor} is unique in law, in particular, unique 
up until first hitting 0. It follows readily that the solution to \eqref{def-wtX} is also unique in law up until first hitting 0, since the conformal map
$f$ is one-to-one and the time change is non-degenerate up until hitting 0.
It then follows that the law of $\wt X$ under the $h$-path transform is 
unique. Since $\wt X$ transformed by \eqref{def-h2} never hits 0, the law of
$\wt X$ is unique without the  restriction of not hitting 0.
As in \cite[page 316]{TW}, this implies $(X,\wt P^{\wt h}_\cdot)$ is 
a strong Markov process and in addition is a Feller process.

If ${\mathcal F}_t$ is the filtration generated by $\wt X$, it is
right continuous; see \cite[Proposition 20.7]{Ba-stoch}. 
Letting $$M_t=\wt \E_z[1_A(\wt X(\wt T_1))\mid  {\mathcal F}_{t\land \wt T_1}],$$
we see that not only is $M_t$ a martingale, but it is also right continuous.
By the strong Markov property,
$$ M_t=\wt \E^{\wt h}_{\wt X_{t\land \wt T_1}}[1_A(\wt X(\wt T_1))]
=\Phi(\wt X_{t\land \wt T_1}).$$

If $S_b=\inf\{t\ge 0: \Phi(\wt X_t)\le b\Phi(x)\}$, Doob's optional stopping theorem tells us that
$$\Phi(x)=\wt \E^{\wt h}_x [\Phi(\wt X_{S_b\land \wt T_1}].$$
If $S_b<\wt T_1$
almost surely, the right hand side 
will be less than or equal to $b\Phi(x)$, a contradiction. 
 Therefore with positive probability there exists an $\omega$
such that $S_b\ge \wt T_1$. The graph of $\wt X_t(\omega)$ will be our
desired $\gamma$.

We now want to apply the support theorem for diffusions (see \cite[pp.\ 25--26]{Ba-deo}) to $\wt B$. The domain in which we apply the support theorem 
is a positive distance from the origin, and hence the drift coefficients of $\wt B$ are bounded there. Let $d$ be the length of $\wt H_1$. Let
$\varphi$ be the curve that moves horizontally to the left a distance $d+1$ and
then horizontally to the right a distance $d+1$ and then vertically a distance
$e+2$, all at unit speed. We start $\wt X$ at a point $z\in \wt H_1$. Let $\eps=1/4$. 
The support theorem says there is probability $p>0$ that $\wt B$ stays within
$\eps$ of this curve until time $2d+e+4$.
Since the local time term only contributes
a push in the horizontal directions and not at all in the vertical directions,
we see that with probability at least  $p$ the process $\wt X$
reflects off $\wt \Gamma_u$, eventually moves right
until it  reflects off $\wt \Gamma_d$, 
and finally exits the strip between $\wt H_{e^{-1}}$ and $\wt H_e$ through
the upper boundary.

Let $F_k$ be the event where the process $\wt X_t$ starts
at a point in $\wt H_{e^{k}}$, there is an excursion from $\wt \Gamma_u$
to $\wt \Gamma_d$, and the process exits the strip between $\wt H_{e^{k-1}}$
and $\wt H_{e^{k+1}}$ for the first time after the excursion is completed, and
the exit is through the top of the strip.
By scaling, $\wt \P^{\wh h}_\cdot (F_k)\ge p$.

By the strong Markov property applied $n-1$ times, 
$$\wt \P^{\wh h}_z(\cup_{k=n+1}^{-1} F_k)=1-\wt P^{\wh h}_z(\cap_{k=n+1}^{-1} F_k^c)\ge 1-q^{-n-1},$$
where $q=1-p$. 
This implies that except for an event of probability at most $q^{-n-1}$, the
curve $\gamma$ must intersect the path of $\wt X_t$.

Let $S_\gamma=\inf\{t: \wh X_t\in \gamma\}$. Then by optional stopping,
\begin{align*}
\Phi(y)&=\wt P^{\wt h}_y(\wt X(\wt T_1)\in A)=\wt E^{\wt h}_y[M_{\wt T_1}]
=\wt \E^{\wt h}_y[M_{\wt T_1\land S_\gamma}]\\
&\ge 
\Big[b\Phi(x)\Big]
\cdot \wh P^{\wh h}_y(S_\gamma<\wt T_1)\\
&\ge b(1-q^{-n-1}) \Phi(x).
\end{align*}
Therefore
$$\P^{\wt h}_y(\wt X(\wt T_1\in A)\ge 
 b(1-q^{-n-1}) \P^{\wt h}_x(\wt X(\wt T_1\in A).$$
Letting $b\uparrow 1$ yields the right hand inequality in \eqref{harn-e1.1}, and reversing the
roles of $x$ and $y$ yields the left hand inequality.
\end{proof}

\begin{corollary}\label{harnack-c1}
Suppose that $ i \geq n+2 $.\\
(i) If $\mu_1$ and $\mu_2$ are two probability measures on $\wt H_{e^{n}}$
and $A$ is a Borel subset of $\wt H_{e^i}$, then
\begin{equation}\label{harnack-e1.2}
-q^{i-n-1}\le \wt P^{\wt h}_{\mu_1}(\wt X(\wt T_{e^i})\in A)
-\wt P^{\wt h}_{\mu_2}(\wt X(\wt T_{e^i})\in A)
\le (1-q^{i-n-1})^{-1}-1.
\end{equation}\\
(ii) If $\mu_1$ and $\mu_2$ are two probability measures on $\wt H_{e^{n}}$,
then the total variation distance between $\wt \P^{\wt h}_{\mu_1}(e^{-i}\wt X(\wt T_{e^i})\in dz)$
and $\wt \P^{\wt h}_{\mu_2}(e^{-i}\wt X(\wt T_{e^i})\in dz)$
tends to 0 as $i\to \infty$.\\
(iii) There exists a unique invariant probability measure for the Markov chain
whose transition densities are given by the kernel $\wt Q$.\\
\end{corollary}

\begin{proof}
(i) Let $A$ be a Borel subset of $\wt H_{e^i}$.
 Multiply all three
terms in \eqref{harn-e1.1} by $\wt P^{\wt h}_y(\wt X(\wt T_{e^i})\allowbreak\in A)$. Then integrate all three
terms with respect to $\mu_1(dx)\times \mu_2(dy)$. Since $\mu_1$ and $  \mu_2$ are
probability measures we obtain
$$(1-q^{i-n-1})\wt P^{\wt h}_{\mu_2}(\wt X(\wt T_{e^i})\in A)
\le \wt P^{\wt h}_{\mu_1}(\wt X(\wt T_{e^i})\in A)
\le (1-q^{i-n-1})^{-1}\wt P^{\wt h}_{\mu_2}(\wt X(\wt T_{e^i})\in A).$$ 
From this we deduce \eqref{harnack-e1.2}.

(ii) follows by replacing $A$ by $e^iA$ in \eqref{harnack-e1.2}
and then letting $i\to \infty$.

(iii) If $\mu_1,\mu_2$ are invariant probability measures and $A$ is a 
Borel subset of $\wt H_1$, then from \eqref{harnack-e1.2} with $A$ replaced by
$e^iA$ and using the invariance of $\mu_1, \mu_2$, we obtain
$$-q^{i-n-1}\le \mu_1(A)-\mu_2(A)\le (1-q^{i-n-1})^{-1}-1.$$
Letting $i\to \infty$, we see that $\mu_1(A)=\mu_2(A)$ for all Borel sets $A$.
\end{proof}

\section{Strong law of large numbers for excursions}\label{sect-slln}

We need a strong law for the number of excursions between $\wt \Gamma_d$ and
$\wt \Gamma_u$. 
Much of this section is devoted to showing our sequence satisfies the hypotheses
of 
a version of the strong law for dependent random variables.     

Define
\begin{align}
S^{d-}_k &= \sup\{t\leq S^d_{k}: X_t\in \Gamma_u\}, \qq k\geq 2,\label{def-Sminus}\\
S^{u-}_k &= \sup\{t\leq S^u_{k}: X_t\in \Gamma_d\}, \qq k\geq 1,\notag
\end{align}
and define $\wt S^{d-}_k, \wt S^{u-}_k, \wh S^{d-}_k, \wh S^{u-}_k$
analogously.
Define
\begin{align}
\wt N^d_{k_1,k_2} &= \#\{j: \wt T_{e^{ k_1}}\leq \wt S^{u-}_j \leq \wt S^u_j < \wt T_{e^{ k_2}}\},
 \qquad   k_1 < k_2,\label{def-wtN}\\
\wt N^u_{k_1,k_2} &= \#\{j: \wt T_{e^{ k_1}}\leq \wt S^{d-}_j \leq \wt S^d_j < \wt T_{e^{ k_2}}\},
 \qquad   k_1 < k_2,\notag\\
\end{align}
and define $\wh N^d_{k_1,k_2}$ and $\wh N^u_{k_1,k_2}$ analogously.
The quantity $\wt N^d_{k_1,k_2}$ counts the number of times
that $\wt S^u_j$ is
between
$\wt T_{e^{ k_1}}$ and $\wt T_{e^{ k_2}}$ and where 
an additional constraint holds:   the excursion from $\wt \Gamma_d$ to
$\wt \Gamma_u$ that ends at time $\wt S^u_j$ 
starts after $\wt T_{e^{ k_1}}$. 

\begin{lemma}\label{s27.1}
(i) For any $m\geq 1$, the random variables $\wt N^d_{k,k+m}$, $k\geq n$, are identically distributed under $\wt \P^{\wt h}_{\wt \nu(e^{n})}$. 
Moreover, the distribution of the sequence $(\wt N^d_{k,k+m}, k\geq i)$ does not depend on $i$.

(ii)
For any $m\geq 1$, the random variables $\wh N^d_{k,k+m}$, $k\geq n$, are identically distributed under $\wh \P^{\wh h}_{\wh \nu(e^{n})}$. 
Moreover, the distribution of the sequence $(\wh N^d_{k,k+m}, k\geq i)$ does not depend on $i$.
\end{lemma}

\begin{proof}
Part (i) follows from the strong Markov property applied at 
time $\wt T(e^{k})$ and Remark \ref{rem-invar2}. Part (ii)  follows from 
conformal invariance and the fact that time changing a process does not affect
the values of $\wh N^u_{k_1,k_2}$ and $\wh N^d_{k_1,k_2}$.
\end{proof}

\begin{lemma}\label{s25.5}
(i) Suppose that $n\leq kj$ and
let $\lambda= \wt \E^{\wt h}_{\wt \nu(e^{n})} \wt N^d_{kj,(k+1)j}$. The mean $\lambda$ depends on $j$ but does not depend on $k$ or $n$.

(ii) For $j\geq 1$ there exists $\sigma^2<\infty$ such that for all $k$ and $n$ such that $kj \geq n$, 
\begin{align*}
 \wt \E^{\wt h}_{\wt \nu(e^{n})}
\left( \left( \wt N^d_{kj,(k+1)j}\right)^2 \right) <\sigma^2.
\end{align*}

(iii)  Fix $j\geq 1$ and for $n\leq kj$ and
 $m\geq 1$  let 
\begin{align*}
\vphi(m) = \left|\wt \E^{\wt h}_{\wt \nu(e^{n})}
 \left[\left( \wt N^d_{kj,(k+1)j}-\lambda\right)
\left(\wt N^d_{(k+m)j,(k+m+1)j}-\lambda\right)\right]
\right|.
\end{align*}
Then 
\begin{align}\label{s30.3}
\sum_{m\geq 1} \vphi(m)/m < \infty.
\end{align}

\end{lemma}

\begin{proof}
(i) follows by Lemma \ref{s27.1} (i). By the same lemma, 
$\sigma^2$ and $\vphi(m)$ do not depend on $n$ (as long as $n\leq kj$)
and we can take $k=0$ in the proofs of (ii)-(iii).
By scaling, the strong Markov property, and the fact that $\wt \nu$ is a
stationary probability distribution, it suffices to take $n=0$. 

(ii)
It is well known that the probability a 3-dimensional Bessel process started
at 1 ever hits the level $e^{-1}$ is equal to $e^{-1}$. Hence the probability
that this process hits the level $e$ before the level $e^{-1}$ is greater
than $1/2$. We apply this to $\wt X^2$. 

Let
\begin{align*}
U_0&=0,\qquad \wt X^2_0=1,\\
U_k&= \inf\{t\geq U_{k-1}: \wt X^2_t/\wt X^2_{U_{k-1}} = 1/2 \text {  or  }2\}.
\end{align*}
Using the strong Markov property repeatedly shows that $\log \wt X^2_{U_k}$ is an asymmetric simple random walk on the integers
started at 0 with positive mean.

Recall Hoeffding's inequality: if $S_n$ is the sum of i.i.d.\ random variables
that are bounded between $a$ and $b$, then
$$\P(|S_n-\E S_n|\ge t)\le 2 \exp(-2t^2/n(b-a)^2).$$ Let $\mu$ denote the
mean of $\log \wt X^2_{U_1}$ and let 
$$J=\inf\{k\ge 0: \log \wt X^2_{U_k}=j\}.$$
By Hoeffding's inequality with $a=-1, b=1$, and scaling, we have
\begin{align*}
\wt \P^{\wt h}_{\wt \nu(1)}(J\ge M)&\le \wt \P^{\wt h}_{\wt \nu(1)}
(\log \wt X^2_{U_M}\le j)
=\P^{\wt h}_{\wt \nu(1)}(\log \wt X^2_{U_M}-M\mu\le j-M\mu)\\
&\le 2\exp(-2(M\mu-j)^2/4M)
\end{align*}
as long as $M\mu\ge j$.

Let $\Lambda_k$ be the number of excursions from $\wt \Gamma_d$ to $\wt \Gamma_u$ that start in $[\wt U_{k-1},\wt U_k)$. 
Look at the time when  $\wt X$ hits the middle line between $\wt \Gamma_d$ and $\wt \Gamma_u$, say, 
at a point $(x,y)$. By the support theorem for diffusions there is probability $p_1>0$, 
independent of $(x,y)$, that $\wt X$ will hit $\wt H_{y/2} \cup \wt H_{2y}$ 
before hitting $\wt \Gamma_d \cup\wt \Gamma_u$.
This and the strong Markov property implies that $\Lambda_k$ has a geometric 
tail:
for some $c_1,c_2>0$ and  $K\geq 1$ and all $k$,
$$\wt \P^{\wt h}_{\wt \nu(1)}( \Lambda_k \geq K)
\leq c_1\exp(-c_2  K) .$$ 

Consider $I$  large enough so that $\sqrt I\ge 2j/\mu$. Let $M=[\sqrt I]$.
If
$\wt N^d_{0,j} \geq I$, then either $J\ge M$ or for some $k \le M$ we
have $\Lambda_k\ge I/M$. Using the above estimates on $J$ and $\Lambda_k$ we
obtain
$$\wt \P^{\wt h}_{\wt \nu(1)} 
(\wt N^d_{0,j} \geq I)
\leq 
2e^{-c_3\sqrt I}+c_4\sqrt I e^{-c_5\sqrt I}.$$
This estimate for the tail probabilities implies that
all moments of $\wt N^d_{0,j}$ under $\wt \P^{\wt h}_{\wt \nu(1)}$ are finite. 

(iii) 
We have
\begin{align}\notag
\vphi(m) &= \left|\wt \E^{\wt h}_{\wt \nu(1)}
 \left[\left( \wt N^d_{0,j}-\lambda\right)
\left(\wt N^d_{mj,(m+1)j}-\lambda\right)\right]
\right|\\
&= \left|\wt \E^{\wt h}_{\wt \nu(1)}
 \left[\left( \wt N^d_{0,j}-\lambda\right)
\wt \E^{\wt h}_{\wt X(\wt T_{e^{ j}})}
\left(\wt N^d_{mj,(m+1)j}-\lambda\right)\right]
\right|\notag\\
&\leq \wt \E^{\wt h}_{\wt \nu(1)}
 \left[\left| \wt N^d_{0,j}-\lambda\right|\cdot
\left|\wt \E^{\wt h}_{\wt X(\wt T_{e^{ j}})}
\left(\wt N^d_{mj,(m+1)j}\right)-\lambda\right|\right].\label{s30.2}
\end{align}
By Proposition \ref{s25.4} and the strong Markov property applied at stopping times $\wt T_{e^{ j}}$ and $\wt T_{e^{ mj}}$, for $m\geq 2$,
\begin{align*}
&\frac
{\wt \E^{\wt h}_{\wt X(\wt T_{e^{ j}})}
\left(\wt N^d_{mj,(m+1)j}\right)}
{\lambda}
=\frac
{\wt \E^{\wt h}_{\wt X(\wt T_{e^{ j}})}
\left(\wt N^d_{mj,(m+1)j}\right)}
{\wt \E^{\wt h}_{\wt \nu(1)}
\left(\wt N^d_{mj,(m+1)j}\right)}\\
&\leq (1-q^{(m-1)j-1} )^{-1}
\frac
{\wt \E^{\wt h}_{\wt \nu(1)}
\left(\wt N^d_{mj,(m+1)j}\right)}
{\wt \E^{\wt h}_{\wt \nu(1)}
\left(\wt N^d_{mj,(m+1)j}\right)}
= (1-q^{(m-1)j-1} )^{-1}.
\end{align*}
Similarly,
\begin{align*}
\frac
{\wt \E^{\wt h}_{\wt X(\wt T_{e^{ j}})}
\left(\wt N^d_{mj,(m+1)j}\right)}
{\lambda}\geq 1-q^{(m-1)j-1} ,
\end{align*}
so 
\begin{align*}
\left|\wt \E^{\wt h}_{\wt X(\wt T_{e^{ j}})}
\left(\wt N^d_{mj,(m+1)j}\right)-\lambda\right|
\leq \lambda c_5 q^{(m-1)j-1}.
\end{align*}
We combine this with \eqref{s30.2} to see that
\begin{align*}
\vphi(m) &\leq \wt \E^{\wt h}_{\wt \nu(1)}
 \left[\left| \wt N^d_{0,j}-\lambda\right|\cdot
\lambda c_5 q^{(m-1)j-1}\right]
\leq c_6 q^{(m-1)j-1}.
\end{align*}
This implies \eqref{s30.3}.
\end{proof}

\begin{lemma}\label{s14.1}
For each $\eps>0$  there exists  $j_0$ such that if $j\ge j_0$ and $n\leq k$ then
\begin{align}\label{s28.1}
\kappa j <\wt \E^{\wt h}_{\wt \nu(e^{n})} \wt N^d_{k,k+j} < (\kappa +\eps) j.
\end{align}
\end{lemma}

\begin{proof}

Let $d$ be the length of $\wh H_b$. Since $\wh D$ is a horizontal strip, $d$ 
does not depend on $b$. By the construction of $\wh D$ we note that $d$ is
finite. 
Let
\begin{align*}
J =J (j)= \min\{m\geq 1: \wh X^1(\wh S^d_m)-\wh X^1(0) \geq j+2d\}.
\end{align*}

By scaling, the strong Markov property, and the fact that $\wh \nu$ is
an invariant measure for $\wt X$, we may take $n=0$.
We will show that for some $c_1<\infty$ and all $j\geq 1$,
\begin{align}\label{o25.6}
0\leq \wh \E^{\wh h}_{\wh \nu(1)}\left(  J \right)
-\wt \E^{\wt h}_{\wt \nu(1)} \wt N^d_{k,k+j}
\leq  c_1.
\end{align}
Note that 
$\wt \E^{\wt h}_{\wt \nu(1)} \wt N^d_{k,k+j} 
= \wh \E^{\wh h}_{\wh \nu(1)} \wh N^d_{k,k+j}$
so the above inequality is equivalent to
\begin{align}\label{o25.7}
0\leq \wh \E^{\wh h}_{\wh \nu(1)}\left(  J \right)
-\wh \E^{\wh h}_{\wh \nu(1)} \wh N^d_{k,k+j}
\leq  c_1.
\end{align}

Since $\wh X_0\in \wh H_{1}$ under $\wh \E^{\wh h}_{\wh \nu(1)}$ and $\wh X^1(\wh S^d_J) \geq j +2d$, we must have $\wh S^d_J \geq \wh T_{e^{ j}}$,
and therefore $J\ge \wh N^d_{k,k+j}$.
This implies the lower bound in \eqref{o25.7}.

By the support theorem for diffusions there exists a  $p_1>0$ such that for every $m$,
if $\wh X_0\in \wh H_{e^{m}}$ then $\wh X$ will hit $\wh \Gamma_d$ before hitting $\wh H_{e^{m+1}}$ with probability greater than $p_1$. Hence, by Lemmas \ref{s27.1} (ii) and \ref{s25.5} (i),
\begin{align*}
\wh \E^{\wh h}_{\wh \nu(1)}\left(  J \right)
-\wh \E^{\wh h}_{\wh \nu(1)} \wh N^d_{k,k+j}
&\leq
\sum _{m\geq 0}(1-p_1)^m \wh \E^{\wh h}_{\wh \nu(1)}
\left( \wh N^{d}_{k+j+m,k+j+m+1}+1\right)\\
&= 
\sum _{m\geq 0}(1-p_1)^m \wh \E^{\wh h}_{\wh \nu(1)}
\left( \wh N^{d}_{k,k+1}+1\right)\\
&= 
\sum _{m\geq 0}(1-p_1)^m \wt \E^{\wt h}_{\wt \nu(1)}
\left( \wt N^{d}_{k,k+1}+1\right)<c_1 <\infty.
\end{align*}
This completes the proof of \eqref{o25.7} and, hence, of \eqref{o25.6}.

In view of Lemma \ref{s25.5}, \eqref{o25.6} implies that 
\begin{align}\label{o25.5}
\wh \E^{\wh h}_{\wh \nu(1)}\left( J  \right)<\infty.
\end{align}.

Note that $\left\{\wh X^1\left(\wh S^d_{k+1}\right) 
- \wh X^1\left(\wh S^d_k\right), k\geq 1\right\}$ is an  i.i.d.\ sequence  under $\wh \P^{\wh h}_{\wh \nu(1)}$, and
\begin{align*}
 \wh X^1(\wh S^d_{J })
=\sum_{k=0}^{J }
 \left(\wh X^1\left(\wh S^d_{k+1}\right) 
- \wh X^1\left(\wh S^d_k\right)\right).
\end{align*}
This, \eqref{o25.5} and Wald's identity imply that
\begin{align}\label{o24.1}
\wh \E^{\wh h}_{\wh \nu(1)}\left(  \wh X^1(\wh S^d_{J })
- \wh X^1\left(0\right)\right)
= \wh \E^{\wh h}_{\wh \nu(1)}\left( J  \right)
\wh \E^{\wh h}_{\wh \nu(1)} \left(\wh X^1\left(\wh S^d_{k+1}\right) 
- \wh X^1\left(\wh S^d_k\right)\right).
\end{align}

It follows from \eqref{o24.1} and Proposition \ref{o21.6} that
\begin{align}\label{o25.1}
\wh \E^{\wh h}_{\wh \nu(1)}\left( J  \right)
= \kappa \wh \E^{\wh h}_{\wh \nu(1)}\left(  \wh X^1(\wh S^d_{J })
- \wh X^1\left(0\right)\right).
\end{align}

Recall that $J $ depends on $j$.
We will show that for every $\eps>0$ there exist $j_1$ and $c_2$ such that for $j>j_1$, 
\begin{align}\label{o25.2}
 j  +2d \leq
\wh \E^{\wh h}_{\wh \nu(1)}\left(  \wh X^1(\wh S^d_{J })
- \wh X^1\left(0\right)\right)
\leq  j +c_2.
\end{align}
The lower bound follows from the definition of $J $.

Using the support theorem for diffusions it is easy to see that there 
exists a  $p_2>0$ such that if $\wh X$ starts from a point in $\wh H_b$ then it will hit $\wh \Gamma_d$ before hitting $\wh H_{b+1}$ with probability greater than $p_2$. By the strong Markov property applied at the
hitting times of $\wh H_{b+k}$, if $\wh X$ starts from a point in $\wh H_b$, it will hit $\wh H_{b+k}$ before hitting $\wh \Gamma_d$ with probability smaller than $(1-p_2)^{k-1}$. Hence, 
\begin{align*}
\wh \E^{\wh h}_{\wh \nu(1)}\left(  \wh X^1(\wh S^d_{J })
- \wh X^1(\wh T_{e^{ j}})\right)&\le 
\sum_{m=0}^\infty (m+1) \P^{\wh h}_{\wh \nu(1)}\left(  \wh X^1(\wh S^d_{J })
- \wh X^1(\wh T_{e^{ j}})\ge m\right)\\
&\le \sum_{m=0}^\infty (m+1)(1-p_2)^{m-1}\le c_2.
\end{align*}
This implies \eqref{o25.2}.

Combining \eqref{o25.1} and \eqref{o25.2} yields for every $\eps>0$ and sufficiently large $j$, 
$$\kappa j<
\kappa(j +2d)\leq
\wh \E^{\wh h}_{\wh \nu(1)}\left(  J \right)
\leq  \kappa(j +c_2)< (\kappa +\eps)j.$$
The lemma follows from this and \eqref{o25.6}.
\end{proof}

\begin{proposition}\label{s25.1}
For each $\eps>0$ and $p<1$ there exists $j_1>0$   such that for all $n\geq j_1$with $n/j_1$ an integer,
$$\wt\P^{\wt h}_{\wt \nu(e^{-n})}(I_1)\ge p,$$ where
\begin{align}\label{o29.1}
I_1=
\bigcap_{k=1}^{n/j_1}
\left\{(\kappa -\eps)kj_1 \leq \wt N^d_{-kj_1,0} \leq (\kappa +\eps)kj_1\right\}.
\end{align}
\end{proposition}

\begin{proof}
Let $\varepsilon\in (0,1)$ and  choose $j>4/\varepsilon>4$ so that the
conclusion of Lemma \ref{s14.1} holds with $\varepsilon$ replaced by
$\varepsilon/4$.
Let $p\in (0,1)$.

Recall that $\lambda= \wt \E^{\wt h}_{\wt \nu(e^{-n})} \wt N^d_{kj,(k+1)j}$ for $-n\le kj$.  The estimate in Lemma \ref{s14.1}
can be written as
\begin{align}\label{o27.3}
\kappa j <\lambda < (\kappa +\eps/4) j.
\end{align}

Recall 
that $\wt N^d_{k_1,k_2} = \#\{i: \wt T_{e^{ k_1}}\leq \wt S^{u-}_i \leq \wt S^u_i < \wt T_{e^{ k_2}}\}$ for
$  k_1 < k_2$.  We also
need  
\begin{equation*}
\wt N^D_{k_1,k_2}=\#\{i: \wt T_{e^{k_1}}\le \wt S^u_i<\wt T_{e^{k_2}}\}.
\end{equation*}
Of course $\wt N^d_{k_1,k_2}\le \wt N^D_{k_1,k_2}$.
They need not be equal since it is possible that $\wt S^{u-}_j< \wt T_{e^{k_1}}$. However if $i_0, i_0+1, \ldots, i_0+n$ are those $i$ for which
$\wt T_{e^{k_1}}\le \wt S^u_i<\wt T_{e^{k_2}}$, then
$\wt S^{u-}_{i_0+\ell}\ge \wt S^u_{i_0}\ge \wt T_{e^{k_1}}$ if $\ell\ge 1$. Hence
$\wt N^d_{k_1,k_2}$ and $\wt N^D_{k_1,k_2}$ can differ by at most 1.
Therefore, for $k> 0$,
\begin{align}\label{d11.1}
\sum _{m=1}^{kj_1/j}  \wt N^d_{-mj,(-m+1)j} \leq
\wt N^d_{-kj_1,0} 
\leq \sum _{m=1}^{kj_1/j}  \wt N^D_{-mj,(-m+1)j}
\leq 
 \sum _{m=1}^{kj_1/j}  \wt N^d_{-mj,(-m+1)j}
+ kj_1/j .
\end{align}

Lemma \ref{s25.5} shows that the sequence $\left\{\wt N^d_{mj,(m+1)j}-\lambda, m\geq 0\right\}$ satisfies the assumptions of  \cite[Cor. 11]{Lyons}, 
a version of the strong law of large numbers for dependent random variables.
The SLLN and \eqref{o27.3} 
imply that
\begin{equation}\label{slln-s102}
\frac{\sum_{m=1}^M\Big[ \wt N^d_{(M-m)j,(M-m+1)j}-\lambda\Big]}{M}
=\frac{\sum_{m=1}^M\Big[ \wt N^d_{(m-1)j,mj}-\lambda\Big]}{M}
\to 0
\end{equation}
almost surely with respect to $\wt \P^{\wt h}_{\wt \nu(e^{-n})}$ provided $n\ge 1$. 
Thus there exists $m_0\ge 1$ such that the left hand side will be less than 
$\varepsilon/2$ in absolute value
for all $M\ge m_0$ with probability at least $p$. This can be written as
\begin{align*}\wt \P^{\wt h}_{\wt \nu(e^{-n})}&\Big(\bigcap_{M=m_0}^{\infty}
\Big\{ (\lambda -\eps/2)M\le \sum_{m=1}^M
 \wt N^d_{(M-m)j,(M-m+1)j} \le (\lambda+\eps/2)M\Big\}\Big)
\ge p.
\end{align*}
Take $j_1=m_0j$.
Letting $M = kj_1/j$ for $k=1, \ldots, n/j_1$,  we have
\begin{align*}
\wt \P^{\wt h}_{\wt \nu(e^{-n})}&\Big(\bigcap_{k=1}^{n/j_1}
\Big\{ (\lambda -\eps/2)kj_1/j\le \sum_{m=1}^{kj_1/j}
 \wt N^d_{(M-m)j,(M-m+1)j}
\le (\lambda+\eps/2)kj_1/j\Big\}\Big)
 \ge p.
\end{align*}
Now use stationarity, scaling, and the strong Markov property to obtain
\begin{align}
\wt \P^{\wt h}_{\wt \nu(e^{-n})}&\Big(\bigcap_{k=1}^{n/j_1}
\Big\{ (\lambda -\eps/2)kj_1/j\le \sum_{m=1}^{kj_1/j}
 \wt N^d_{-mj,(-m+1)j}
\le (\lambda+\eps/2)kj_1/j\Big\}\Big) \ge p\label{d11.2}
\end{align}
if $n/j_1$ is an integer larger than 1.

By \eqref{o27.3}, using $j>4$,
\begin{align*}
(\kappa -\eps/2)kj_1 &\leq (\lambda/j -\eps/2)kj_1 = (\lambda -j\eps /2)kj_1/j
\leq (\lambda -\eps/2 )kj_1/j,\\
(\kappa +\eps/2)kj_1 &\geq (\lambda/j-\eps/4 +\eps/2)kj_1 = (\lambda +j\eps /4)kj_1/j
\geq (\lambda +\eps/2 )kj_1/j.
\end{align*}
Hence, \eqref{d11.2} implies
\begin{align}
\wt \P^{\wt h}_{\wt \nu(e^{-n})}&\Big(\bigcap_{k=1}^{n/j_1}
\Big\{ (\kappa -\eps/2)kj_1\le \sum_{m=1}^{kj_1/j}
 \wt N^d_{-mj,(-m+1)j}
\le (\kappa+\eps/2)kj_1\Big\}\Big) \ge p.\label{j12.1}
\end{align}

Since $1/j < \eps/4$, we have for $k>0$,
\begin{align*}
(\kappa +\eps/2)kj_1  + kj_1/j 
&\leq (\kappa +3\eps/4)kj_1.
\end{align*}
We combine this, \eqref{d11.1} and \eqref{j12.1} to obtain
\begin{align*}
\wt\P^{\wt h}_{\wt \nu(e^{-n})}\left(
\bigcap_{k=1}^{n/j_1}
\left\{(\kappa -\eps/2)kj_1 \leq \wt N^d_{-kj_1,0} 
 \leq (\kappa +3\eps/4)kj_1\right\} \right) \geq p.
\end{align*}
This completes the proof.
\end{proof}

\section{Back to the quadrant}\label{sect-quadrant}

We have used $\wt D$ and $\wh D$ to establish a number of results. In this
section we convert these to the corresponding results for $D$.

Recall the definitions of $\wt H_b, H_b, \wt T_b$, and $T_b$ given in
\eqref{def-H} and \eqref{def-Tb}. Define 
\begin{equation}\label{def-h1}
h(z)=\wt h(\F(z))
\end{equation}
and 
\begin{equation}\label{def-nu}
\nu(e^{n})(A)=\wt \nu(e^{n})(\F(A)),
\end{equation}
where $A$ is a Borel subset of $H_{e^{n}}$ and $\F(A)=\{\F(z):z\in A\}$.

An easy calculation shows that the imaginary part of the analytic function
$\F$ is
$$ \Phi(re^{i\theta})=r^\alpha\cos(\alpha \theta-\theta_1).$$
By the definition of $\wt h$ we have that
$\wt h(w)$ is the imaginary part of $w$. Therefore
$h(z)=\wt h(\F(z))$ is the imaginary part of $\F(z)$, and hence 
\begin{equation}\label{formula-h}
h(re^{i\theta})
=\Phi(re^{i\theta})=r^\alpha\cos(\alpha \theta-\theta_1).
\end{equation} 

The function $h$ is harmonic in the interior of $D$ since it
is the imaginary part of an analytic function.  
By the calculation of \cite[(2.11)]{VW}, $\nabla \Phi\cdot {\bf v}=0$
on the boundary of $D$. Using Ito's formula, we see that $h(X_t)$ is a
continuous martingale when $X$ is a solution to \eqref{def-skor}. 

Define $\P^h_z$ to be the $h$-path transform of $\P$, defined analogously to 
\eqref{def-hpath}. Thus if $S$ is a stopping time relative to a filtration
$\{{\calF}_t\}$ with
respect to which $X$ is adapted and $A \in {\calF}_S$, then
\begin{equation}\label{def-Ph}
\P^h_x(A)=\E^h_z[h(X_S);A]/h(z).
\end{equation}
Since $h(X_t)$ is a time change of one-dimensional 
Brownian motion, 
$\P_z(T_a<\infty)=1$ if $z\in H_b$ and $0<a<b$.
Then 
\begin{equation}\label{quad-s1}
\P_z^h(T_a<\infty)=\E^z[h(X(T_a)); T_a<\infty]/h(z)=a/b.
\end{equation}

\begin{lemma}\label{s16.1}
For each  $\eps_1>0$
there exist  $n_1$ and  $ j_0$ such that for all $n\geq n_1$ and $j_1\geq j_0$,
$$\P^h_{\nu(e^{-n})} (G)\le 2e^{-j_1\eps_1},$$
where
\begin{align}\label{o31.5}
G=
\bigcup_{m=1}^{n/j_1}
\left\{\inf_{t\geq T(e^{-mj_1})} h(X_t) \leq 
e^{-mj_1(1+\eps_1) }\right\}.
\end{align}
\end{lemma}

\begin{proof}
By \eqref{quad-s1}, if $z\in H_b$  and $0<a<b$ then
 $\P^{ h}_z( T_a< \infty) = a/b$,
so 
\begin{align}\label{j13.1}
\P^h_{\nu(e^{-n})} \left(
\inf_{t\geq T(e^{-mj_1})} h(X_t) \leq 
e^{-mj_1(1+\eps_1) }
\right)
\leq e^{-mj_1(1+\eps_1) }/e^{-mj_1}
= e^{-mj_1\eps_1}.
\end{align}
This leads to
\begin{equation}\label{quad-e5.1}\P^h_{\nu(e^{-n})}(G) \leq 
 \sum_{m=1}^{n/j_1} e^{-mj_1\eps_1}
\leq \frac{e^{-j_1\eps_1}}{1-e^{-j_1\eps_1}}\le 2e^{-j_1\eps_1}
\end{equation}
if $j_1$ is
sufficiently large. 
\end{proof}

Let $T$ be a finite stopping time for $X$ and let 
$$U(T,t)=\inf_{T\le s\le T+t}(B_s^2-B_T^2), 
\qquad V(T,t)=\sup_{T\le s\le T+t}|X_s-X_T|.$$
Let 
\begin{align}
F_1(T)&=\Big\{ U(T,e^{-2mj_1(1+3\eps_1)/\alpha})\le -e^{-mj_1(1+4\eps_1)/\alpha}
\Big\},\label{n3.4}\\
F_2(T)&=\Big\{V(T,e^{-2mj_1(1+3\eps_1)/\alpha})\le \chi_1e^{-mj_1(1+2\eps_1)/\alpha}\Big\},\label{j14.2}
\end{align}
where $\chi_1$ is the constant from Proposition \ref{prelim-Lcompletely}.

\begin{lemma}\label{n3.6} There exists a positive constant $c_1$ 
such that if $h(X_T)\ge e^{-mj_1}$ and $\eps_1>0$, then 
for all positive $m,n,j_1$ we have
\begin{align}
\P^h_{\nu(e^{-n})}&(F_1(T)^c)\le c_1e^{-mj_1\eps_1}, 
\label{n.37a}\\
\P^h_{\nu(e^{-n})}&(F_2(T)^c)\le c_1e^{-mj_1\eps_1}.\label{n.37b}
\end{align}
\end{lemma}

\begin{proof} 
By the strong Markov property it suffices to prove our result if we take $T$ 
to be identically equal to 0
and we replace $\P^h_{\nu(e^{-n})}$ by $\P^h_z$ in \eqref{n.37a} and 
\eqref{n.37b}, with $h(z)\ge e^{-mj_1}$.

Let
\begin{align}\label{j13.2}
A_1=e^{-mj_1(1+\eps_1)}, \qquad 
&t_1=e^{-2mj_1(1+3\eps_1)/\alpha},\\
\lambda_1= e^{-mj_1(1+4\eps_1)/\alpha}, \qquad
&\lambda_2= e^{-mj_1(1+2\eps_1)/\alpha}.\notag
\end{align}
Let $G_1=\{\inf_{t\ge 0} h(X_t)\le A_1\}$.
By our choice of $A_1$ and the argument in \eqref{j13.1},
\begin{equation}\label{quad-e1.1}
\P^h_z(G_1)\le A_1/h(z)\le e^{-mj_1\eps_1}.
\end{equation}

We have $h(x)\le |x|^\alpha$ by \eqref{formula-h}.
A standard calculation shows that $|\nabla h(x)|/h(x)\le c_2/|x|$. Hence
on the event $G_1^c$, for all $t\geq 0$,
\begin{align}\label{j13.3}
\frac{|\nabla h(X_t)|}{h(X_t)}\le \frac{c_2}{|X_t|}\le \frac{c_2}
{h(X_t)^{1/\alpha}} \le c_2e^{mj_1(1+\eps_1)/\alpha}.
\end{align}

The process $X_t$ solves the Skorokhod equation, where $B_t$ is a 2-dimensional 
Brownian motion under $\P_z$. By \cite[pp. 61--62]{Ba-pta},
the process $W_t=B_t-D_t$ is a Brownian motion under $\P^h_z$, where
$$D_t=\int_0^t \frac{\nabla h(X_s)}{h(X_s)}\, ds.$$
Use \eqref{j13.2} and \eqref{j13.3} to see that
on the event $G_1^c$ we have
\begin{equation}\label{quad-e1.2}
\sup_{s\le t_1}|D_{s}|\le c_2e^{mj_1(1+\eps_1)/\alpha} t_1= c_2e^{-mj_1(1+5\eps_1)/\alpha}.
\end{equation}  

By Brownian scaling and standard estimates for Brownian motion,
\begin{equation}\label{quad-e1.4}
\P_z^h\Big(\inf_{s\le t_1}(W^2_s-W^2_0)\ge -2\lambda_1\Big)\le c_{3}\lambda_1/\sqrt{t_1}\le c_{4} e^{-mj_1\eps_1}
\end{equation}
and
\begin{equation}\label{quad-e1.3}
\P^h_z\Big(\sup_{s\le t_1} |W_s-W_0|\ge \tfrac12 \lambda_2\Big)\le 2e^{-\lambda_2^2/8t_1}
\le 2\exp\Big(-e^{2\eps_1mj_1/\alpha}/8\Big).
\end{equation}

Suppose $mj_1\eps_1$ is sufficiently large that the right hand side of \eqref{quad-e1.2} is smaller 
than $\lambda_1$. We have $D_0=0$, hence
$W_0=B_0$, and thus  on $G_1^c$  we have
\begin{align*}
\Big|\inf_{s\le t_1}(W^2_s-W^2_0)-\inf_{s\le t_1} (B^2_s-B^2_0)\Big|
\le \sup_{s\le t_1}|D_s|\le c_2e^{-mj_1(1+5\eps_1)/\alpha} \le \lambda_1.
\end{align*}
If the event in \eqref{quad-e1.4} does not hold, that is, 
if $\inf_{s\le t_1}(W^2_s-W^2_0)< -2\lambda_1$, then the above inequality implies that
\begin{align*}
\inf_{s\le t_1} (B^2_s-B^2_0)
\leq \inf_{s\le t_1}(W^2_s-W^2_0) +\lambda_1 < - \lambda_1.
\end{align*}
This means that $F_1(0)$ holds.
In view of \eqref{quad-e1.1} and \eqref{quad-e1.4},
\begin{align*}
\P^h_z( F_1(0)^c)
\leq \P^h_z(G_1)
+\P^h_z(G_1^c\cap F_1(0)^c)
\le c_5e^{-mj_1\eps_1/\alpha}
\end{align*}
 and \eqref{n.37a}
follows for $mj_1\eps_1$ sufficiently large. If $mj_1\eps_1$ is small, \eqref{n.37a}
follows by adjusting $c_1$.

Similarly suppose $mj_1\eps_1$ is sufficiently large that 
the right hand side of  \eqref{quad-e1.2} is smaller 
than $\lambda_2/2$. 
Then on $G^1_c$ we have
$$\Big|\sup_{s\le t_1}|B_s-B_0|-\sup_{s\le t_1} |W_s-W_0|\Big|
\le \sup_{s\le t_1}|D_s|\le c_2e^{-mj_1(1+5\eps_1)/\alpha} \le \lambda_2/2$$
if $mj_1\eps_1$ is sufficiently large. 
If the event in \eqref{quad-e1.3} does not hold, that is, if $\inf_{s\le t_1}|W_s-W_0|< \lambda_2/2$,
then the above inequality implies that
\begin{align*}
\inf_{s\le t_1} |B_s-B_0|
\leq \inf_{s\le t_1}|W_s-W_0| +\lambda_2/2 <  \lambda_2.
\end{align*}
Therefore using Proposition \ref{prelim-Lcompletely}, \eqref{quad-e1.1} and \eqref{quad-e1.3},
\begin{align*}
\P^h_z( F_2(0)^c)
&\leq \P^h_z(G_1)
+\P^h_z(G_1^c\cap F_2(0)^c)\\
&= \P^h_z(G_1)
+\P^h_z\Big(G_1^c\cap \Big\{\sup_{s\le t_1}|X_s-X_0|>\chi_1\lambda_2\Big\}\Big)\\
&\le \P^h_z(G_1)
+\P^h_z\Big(G^c_1\cap \Big\{\sup_{s\le t_1}|B_s-B_0|>\lambda_2\Big\}\Big)\\
&\leq e^{-mj_1\eps_1} + 2\exp\Big(-e^{2\eps_1mj_1/\alpha}/8\Big),
\end{align*}
and \eqref{n.37b} follows.
\end{proof}

\begin{remark}\label{R-1-dim}
Suppose $t_0>0$,  $f:[0,t_0]\to \R^2$ is continuous with $f(0)\in D\setminus\{0\}$ and $g:[0,t_0]
\to D$ satisfies
$$g(t)=f(t)+\int_0^t {\bf v}(g(s))\, dL^g_s$$
for all $t$ up until the first time $g$ hits the origin,
where $L^g$ is a continuous non-decreasing function with $L^g_0=0$ and 
$L^g$ increases only when $g\in \partial D$.
Uniqueness holds for this problem. 

To see this, let $g$ be any such solution. Note that  $g$ will move in tandem with $f$ when
$g$ is in the interior of $D$. If $g(s)\in \Gamma_d$, then it is easy
to check that
\begin{align*}
L_g^2(s+\eps)-L_g^2(s)&=\sup_{s\le r\le s+\eps}\Big[-(f^2(r)-f^2(s))^+\Big],\\
g^2(s+\eps)-g^2(s)&=f^2(s+\eps)-f^2(s) +(L_g^2(s+\eps)-L_g^2(s)),\\
g^1(s+\eps)-g^1(s)&=f^1(s+\eps)-f^1(s) -a_1(L_g^2(s+\eps)-L_g^2(s)),
\end{align*}
provided $\eps$ is small enough so that $g(r)$ does not hit $\Gamma_u$ in the
time interval $[s,s+\eps]$. 
Similar formulas hold when $g(s)\in \Gamma_u$. These imply that if there is
uniqueness up to time $s$ and $g(s)\ne 0$, there will be uniqueness for
all times $[s,s+\eps]$ if $\eps$ is sufficiently small (depending on $s$).
From these we deduce  uniqueness for all $s$ up until the first time
$g$ hits 0. 
Unlike the stochastic case, there is no difficulty with the ``piecing together argument,'' 
because there is no randomness here.

For the corresponding result for stochastic processes, use the above result
for each path separately.
\end{remark}

For $t_1>0$ and $f$ a continuous function from $[0,t_1]$ to $\R^2$, define
\begin{equation}\label{def-E}
\Osc(f,t_1,y)=\sup_{0\le s_1\le s_2\le s_1+y\le t_1} |f(s_2)-f(s_1)|.
\end{equation}

\begin{lemma}\label{C2-lemma}
Suppose $t_1,c_1>0$, $f:[0,t_1]\to \R^2$ is continuous, and $g_1$ and $g_2$ are
two solutions to \eqref{prelim-e100.1} with $g_1(0), g_2(0)\ne 0$, $|g_1(0)-g_2(0)|\le c_1/4$, and $\sup_{s\le t_*\land t_1} |g_1(s)-g_2(s)|\ge c_1$. Here
$t_*=\inf\{t>0: g_2(t)=0\}$. Let $\chi_1$ be defined by
Proposition \ref{prelim-Lcompletely}.
Then 
$$t_*\ge \inf\{y: \Osc(f,t_1,y)\ge c_1/4\chi_1\}.$$
\end{lemma}

\begin{proof}
Let $u=\inf\{t: |g_1(t)-g_2(t)|\ge 3c_1/4\}$ and note that $u\leq t_*$. Then
$$|(g_1(u)-g_1(0))-(g_2(u)-g_2(0))|\ge c_1/2.$$
By Proposition \ref{prelim-Lcompletely}(ii)
$$|g_i(u)-g_i(0)|\le \chi_1 \Osc(f,t_1,u), \qquad i=1,2.$$
Therefore $c_1/2\le 2\chi_1\Osc(f,t_1,u)$.
We have $u\le t_*$ and $\Osc(f,t_1,u)\ge c_1/4\chi_1$. Our result 
follows.
\end{proof}

The following is the key estimate of the paper.

\begin{theorem}\label{o28.1}

Suppose $a_1>0$, $a_2<0$, $\alpha>0$, $\beta>1$, and $\psi>1/\alpha$.
Let $X$ be a solution to \eqref{def-skor} with initial distribution $\P^h_{\nu(e^{-n})}$.
Suppose $p_1\in (0,1)$.  There exist  $c_1>0$ and $n_1$ such that if $n\geq n_1$,
then the following hold.

(i) There exist a random variable $Y_0\ne 0$ that is measurable with respect to the $\sigma$-field $\sigma\{X_t, t\geq 0\}$ 
and a continuous solution to 
\begin{align}\label{o27.2}
Y_t &= Y_0 + B_t  + \int_0^t  \bv(Y_s)\, dL^Y_s, \qquad 0\leq t \leq T_*,
\end{align}
where $T_* = \inf\{t\geq 0: Y_t =0\}$.

(ii) $$\P^h_{\nu(e^{-n})}(C_i)>p_1, \qquad i=1,2,3,$$ where
\begin{align}
&C_1=\left\{|X_0-Y_0| \leq e^{-n}\right\},\label{d26.1}\\
&C_2=\left\{T_* \geq \inf\{y: \Osc(B,T_1,y)\ge c_1/4\chi_1\} \right\},
\notag\\       
&C_3=\left\{\sup_{0\leq t\leq T_1\land T_*} |X_t-Y_t| \geq c_1\right\}.
\notag       
\end{align}
\end{theorem}

\begin{proof}
The random variable $Y_0$ will be defined using information from the future of $X$, 
but this does not pose a problem with the solution to \eqref{o27.2}. 
By Proposition \ref{prelim-Lcompletely}(i)
this equation has a solution in the deterministic sense and therefore for every $\omega$ such that $B_t$ is continuous. 
By Remark \ref{R-1-dim}, the solution will be unique at least up until time 
$T_*$.
    
\emph{Step 1.} We begin by setting some parameters and making some definitions.
By the formula for $h$ in \eqref{formula-h} there exists $d>0$ not depending on $b$ such that $|x| \geq d b^{1/\alpha}$ for $x\in T_b$.
Let $\eps>0$ and set
\begin{equation}\label{quad-def-rho}
\rho=\psi/\kappa=\log|\tan \theta_1|+\log|\tan \theta_2|.
\end{equation}
 Write $\eps_1 = \eps\rho$. We can take  $\eps$ small enough so that
$(1+4\eps_1)/\alpha < \psi - \eps_1$ and $(1+\eps)/\alpha < \psi - \eps_1$.
We choose $\eps>0$ smaller if necessary so  that $\rho (\kappa -\eps) >1$.
We take  $m_1$ large enough that for $j_1\geq1$ and $m\geq m_1$,
\begin{align}
 e^{-(\psi -\eps_1)mj_1} &\leq  e^{-(m+1)j_1(1+4\eps_1)/\alpha} ,\label{d24.1}\\
 e^{-(\psi -\eps_1)mj_1} &< d e^{-(m+1)j_1(1+\eps)/\alpha}.\label{j14.5}     
\end{align}
We assume that $n$ is so large that $n/j_1 > m_1$.

Let ${\Theta}= \sup\{k: S^d_k<  T_1\}$, $K = \inf\{k: S^d_k > T_{e^{-m_1j_1}}\}$ and
 $Y_0 = X_0 + (0, e^{-\rho {\Theta}})$.

We will write $S^d_{X,k}=S^d_k$ and $S^u_{X,k}=S^u_k$ to emphasize the
dependence on $X$, 
where $S^d_k$ and $S^u_k$ are defined in \eqref{def-S}.  We will use the analogous notation for $Y$:
 \begin{align*}
S^u_{Y,0} &=0,\qquad S^d_{Y,0}=0,\\
S^d_{Y,k} &= \inf\{t\geq S^u_{Y,k-1}: Y_t\in \Gamma_d\}, \qq k\geq 1,\\
S^u_{Y,k} &= \inf\{t\geq S^d_{Y,k}: Y_t\in \Gamma_u\}, \qq k\geq 1.
\end{align*}

\emph{Step 2.} 
Let
\begin{align}\label{n5.3}
A_k&=
\{S^d_{X,k} < S^d_{Y,k} < S^u_{X,k} < S^u_{Y,k}
< S^d_{Y,k+1} < S^d_{X,k+1} < S^u_{Y,k+1} < S^u_{X,k+1} < S^d_{X,k+2}\},\\
E^k &= \left\{Y_t \ne 0, 0\leq t \leq S^d_{X,k+2}\right\}.\notag
\end{align}

\begin{figure} [ht]
\includegraphics[width=0.5\linewidth]{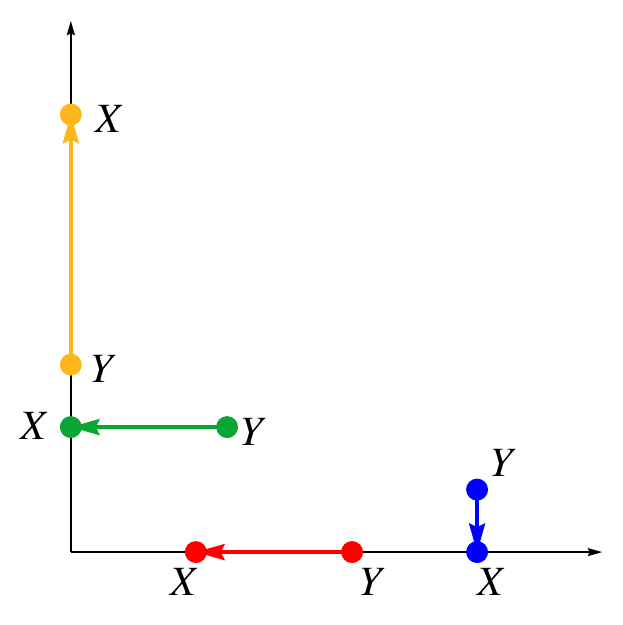}
\caption{ Positions of $X$ and $Y$ at different  times $t$.
(a) blue, $t=S^d_{X,k}$; (b) red, $t= S^d_{Y,k} $; (c) green, $t=S^u_{X,k} $;
(d) orange, $t= S^u_{Y,k}$. }
\label{fig11}
\end{figure}

Equations \eqref{def-skor} and \eqref{o27.2}  show that $X$ and $Y$ move in tandem until one of the two processes hits $\partial D$. Our assumptions 
that $Y_0 = X_0 + (0, e^{-\rho {\Theta}})$, $\theta_1>0$, and $\theta_2<0$ 
imply that if $A_1\cap E^1$ holds, then the following is true for $k=1$:
(The reader may wish to consult the last few paragraphs of Section \ref{sect-prelim} at this point.)
\begin{align}\label{n5.4}
X^1(S^d_{X,k})&=Y^1(S^d_{X,k}) 
\text{  and  } X^2(S^d_{X,k})<Y^2(S^d_{X,k}),\\
X^1(S^d_{Y,k})&<Y^1(S^d_{Y,k}) 
\text{  and  } X^2(S^d_{Y,k})=Y^2(S^d_{Y,k}),\notag\\
X^1(S^u_{X,k})&<Y^1(S^u_{X,k}) 
\text{  and  } X^2(S^u_{X,k})=Y^2(S^u_{X,k}),\notag\\
X^1(S^u_{Y,k})&=Y^1(S^u_{Y,k}) 
\text{  and  } X^2(S^u_{Y,k})>Y^2(S^u_{Y,k}),\notag\\
X^1(S^d_{Y,k+1})&=Y^1(S^d_{Y,k+1}) 
\text{  and  } X^2(S^d_{Y,k+1})>Y^2(S^d_{Y,k+1}),\notag\\
X^1(S^d_{X,k+1})&>Y^1(S^d_{X,k+1}) 
\text{  and  } X^2(S^d_{X,k+1})=Y^2(S^d_{X,k+1}),\notag\\
X^1(S^u_{Y,k+1})&>Y^1(S^u_{Y,k+1}) 
\text{  and  } X^2(S^u_{Y,k+1})=Y^2(S^u_{Y,k+1}),\notag\\
X^1(S^u_{X,k+1})&=Y^1(S^u_{X,k+1}) 
\text{  and  } X^2(S^u_{X,k+1})<Y^2(S^u_{X,k+1}),\notag\\
X^1(S^d_{X,k+2})&=Y^1(S^d_{X,k+2}) 
\text{  and  } X^2(S^d_{X,k+2})<Y^2(S^d_{X,k+2}).\label{n5.5}
\end{align}
Moreover,
\begin{align}\label{n5.6}
X(S^d_{Y,k})-Y(S^d_{Y,k}) &= X(S^u_{X,k})-Y(S^u_{X,k}),\\
X(S^u_{Y,k})-Y(S^u_{Y,k}) &= X(S^d_{Y,k+1})-Y(S^d_{Y,k+1}),\notag\\
X(S^d_{X,k+1})-Y(S^d_{X,k+1}) &= X(S^u_{Y,k+1})-Y(S^u_{Y,k+1}),\notag\\
X(S^u_{X,k+1})-Y(S^u_{X,k+1}) &= X(S^d_{X,k+2})-Y(S^d_{X,k+2}).\label{n5.7}
\end{align}
We also have
\begin{align}\label{n5.8}
|X(S^d_{Y,k})-Y(S^d_{Y,k})|
& = |X(S^d_{X,k})-Y(S^d_{X,k})|\cdot| \tan\theta_1|,\\
|X(S^u_{Y,k})-Y(S^u_{Y,k})|
& = |X(S^u_{X,k})-Y(S^u_{X,k})|\cdot| \tan\theta_2|,\notag\\
|X(S^d_{X,k+1})-Y(S^d_{X,k+1})|
& = |X(S^d_{Y,k+1})-Y(S^d_{Y,k+1})|\cdot| \tan\theta_1|,\notag\\
|X(S^u_{X,k+1})-Y(S^u_{X,k+1})|
& = |X(S^u_{Y,k+1})-Y(S^u_{Y,k+1})|\cdot| \tan\theta_2|,\notag\\
|X(S^d_{X,k+2})-Y(S^d_{X,k+2})|
& = |X(S^d_{X,k})-Y(S^d_{X,k})| e^{2\rho}.\label{n5.9}
\end{align}

The conditions in \eqref{n5.5} are the same as in \eqref{n5.4} except that $k$ is replaced by $k+2$. The cycle depicted in Fig. \ref{fig11}
is repeated twice, once as shown, and then with the roles of $X$ and $Y$ interchanged. At the end of the two cycles, the configuration of $X$ and $Y$ is similar to the one at the beginning of the cycle.
We have already noticed that if $A_1\cap E^i$ holds, then \eqref{n5.4}-\eqref{n5.9} are true for $k=1$.
An induction argument shows that
 if $A_i\cap E^i$ holds  for all odd $i\leq m$, then \eqref{n5.4}-\eqref{n5.9} hold for all odd $k\leq m$. It then follows from \eqref{n5.9} and an induction argument that 
\begin{align}\label{j14.1}
|X(S^d_{X,m})-Y(S^d_{X,m})|= e^{-\rho {\Theta} +m\rho}.
\end{align}

\emph{Step 3.} 
Recall the definitions of $I_1$ from \eqref{o29.1} and
 $G$ from  \eqref{o31.5}. Assume that $I_1 \cap G^c$ holds.

Consider an odd $k\le K-1$ and let $m$ be such that
\begin{equation}\label{d21.2a}
(\kappa -\eps)mj_1 \leq {\Theta}- k \leq (\kappa -\eps)(m+1)j_1.
\end{equation}

Assume that $A_i\cap E^i$ holds for all odd $i<k$, and therefore
 \eqref{n5.9} holds  (with $i$ in place of $k$) for all odd $i<k$.  
We use \eqref{j14.1} and \eqref{d21.2a} to see that
\begin{align}\notag
|X(S^d_{X,k})-Y(S^d_{X,k})|&
= |X(0)-Y(0)| e^{2\rho k/2}
\leq
|X(0)-Y(0)| e^{\rho({\Theta}- (\kappa -\eps)mj_1)}\\
&=e^{-\rho {\Theta}} e^{\rho({\Theta}- (\kappa -\eps)mj_1)}
=e^{-(\psi -\eps_1)mj_1}.\label{d21.3}
\end{align}

Since $I_1$ holds, for  $i\leq n/j_1 $, 
\begin{align*}
(\kappa -\eps)ij_1 \leq \wt N^d_{-ij_1,0} \leq (\kappa +\eps)ij_1.
\end{align*}
Neither a time change nor a conformal mapping change the validity of these
inequalities, so we have
\begin{align}\label{j15.1}
(\kappa -\eps)ij_1 \leq  N^d_{-ij_1,0} \leq (\kappa +\eps)ij_1.
\end{align}
We apply the left hand side inequality with $i=m+1$ and 
\eqref{d21.2a} to see that $S^d_{X,k} \geq  T_{e^{-(m+1)j_1}}$.
By Lemma \ref{s16.1}, \eqref{j14.5} and \eqref{d21.3}, 
\begin{align*}
|X(S^d_{X,k})| > d\, e^{-(m+1)j_1(1+\eps)/\alpha} 
> e^{-(\psi -\eps_1)mj_1} \geq |X(S^d_{X,k})-Y(S^d_{X,k})|,
\end{align*}
where $d$ was defined at the beginning of the proof. 
We will show below that the inequality extends to $t\in[S^d_{X,k},S^d_{X,k+1}]$,
that is, $|X_t | > |X_t - Y_t|$. It follows that $Y$ does not visit the origin, showing the event $E^k$ holds.

We will argue that $S^d_{Y,k} < S^u_{X,k}\land S^u_{Y,k}$ 
if the events $F_1(T)$ and $F_2(T)$ defined in \eqref{n3.4}-\eqref{j14.2} hold with $T=S^d_{X,k}$.
We will analyze the motions of $X$ and $Y$ on the interval $[S^d_{X,k},S^d_{Y,k}]$.

First, recall that $\theta_1>0$ and $X^1(S^d_{X,k})=Y^1(S^d_{X,k}) $. It follows that
$X^1_t \leq Y^1_t$ for $t\in[ S^d_{X,k},S^u_{Y,k}]$ and therefore 
$ S^u_{X,k}\leq S^u_{Y,k}$.

Next we show that $S^d_{Y,k} < S^u_{X,k}$.
Since $S^d_{X,k} \geq  T_{e^{-(m+1)j_1}}$ and $F_2(S^d_{X,k})$ holds, we have
\begin{align}\label{d21.5}
\sup\left\{ \left|X_t -X(S^d_{X,k})\right|
: S^d_{X,k}\leq t \leq S^d_{X,k} +e^{-2 (m+1) j_1(1+3\eps_1) /\alpha} \right\}
\leq e^{- (m+1) j_1(1 +2 \eps_1)/\alpha} .
\end{align}
Recall that $S^d_{X,k} \geq  T_{e^{-(m+1)j_1}}$ and $G^c$ holds. Thus
$|X(S^d_{X,k})|\geq de^{-(m+1)j_1(1+\eps_1)/\alpha}$. This and \eqref{d21.5} imply that for $m\geq 1$ and sufficiently large $j_1$,
\begin{align*}
\inf&\left\{ X^1_t: S^d_{X,k}\leq t \leq S^d_{X,k} +e^{-2 (m+1) j_1(1+3\eps_1) /\alpha} \right\}\\
&\geq d e^{-(m+1)j_1(1+\eps_1)/\alpha}- e^{- (m+1) j_1(1 +2 \eps_1)/\alpha} > 0 .
\end{align*}
Therefore 
\begin{align}\label{d22.2}
S^u_{X,k}>S^d_{X,k} +e^{-2 (m+1) j_1(1+3\eps_1)\psi/\alpha}.
\end{align}

Since $X(S^d_{X,k}) \in \Gamma_d$, \eqref{d21.3} implies that $Y^2(S^d_{X,k}) \leq e^{-(\psi -\eps_1)mj_1}$.
In view of \eqref{d24.1}, $Y^2(S^d_{X,k}) \leq  e^{- (m+1) j_1(1+4\eps_1)/\alpha}$.
This and the event $F_1(S^d_{X,k})$ holding imply that
\begin{align*}
-Y^2(S^d_{X,k})&=Y^2(S^d_{Y,k}) -Y^2(S^d_{X,k})\\
&\leq\inf\left\{ Y^2_t -Y^2(S^d_{X,k})
: S^d_{X,k}\leq t \leq \min(S^d_{X,k} +e^{-2 (m+1) j_1(1+3\eps_1)  /\alpha},
S^d_{Y,k}) \right\}\\
&=\inf\left\{ B^2_t -B^2(S^d_{X,k})
: S^d_{X,k}\leq t \leq \min(S^d_{X,k} +e^{-2 (m+1) j_1(1+3\eps_1)  /\alpha},
S^d_{Y,k}) \right\}\\
&= \max(- e^{- (m+1) j_1(1+4\eps_1)/\alpha} ,B^2(S^d_{Y,k}) -B^2(S^d_{X,k}))\\
&= \max(- e^{- (m+1) j_1(1+4\eps_1)/\alpha} ,Y^2(S^d_{Y,k}) -Y^2(S^d_{X,k}))\\
&= \max(- e^{- (m+1) j_1(1+4\eps_1)/\alpha} , -Y^2(S^d_{X,k}))=-Y^2(S^d_{X,k}).
\end{align*}
Thus the  inequality is in fact an equality and, therefore,
\begin{align*}
S^d_{Y,k}\leq S^d_{X,k} +e^{-2 (m+1) j_1(1+3\eps_1)  /\alpha}.
\end{align*}
This and \eqref{d22.2} imply that $S^d_{Y,k} < S^u_{X,k}$, as claimed.
We conclude that $E^k$ holds. 

\emph{Step 4.} 
Suppose that $I_1\cap G^c$ holds and also  $A_k$ holds for each odd $k\leq K-1$. 
We apply \eqref{j15.1} with $i=m_1$ to obtain,
\begin{align*}
(\kappa -\eps)m_1j_1 \leq  N^d_{-m_1j_1,0} \leq (\kappa +\eps)m_1j_1.
\end{align*}
Then ${\Theta}-K \leq (\kappa + \eps) m_1j_1$, and in view of \eqref{j14.1},
\begin{align*}
| X(S^d_{X,K})-Y(S^d_{X,K})|= e^{-\rho {\Theta} +K\rho} \geq e^{-\rho(\kappa + \eps) m_1j_1} .
\end{align*}
We conclude that $C_3$ in \eqref{d26.1} holds if we set $c_1 =  e^{-\rho(\kappa + \eps)m_1 j_1} $.

Let $m_2$ be such that
\begin{equation}\label{d21.2b}
(\kappa -\eps)m_2j_1 \leq {\Theta} \leq (\kappa -\eps)(m_2+1)j_1.
\end{equation}
We are assuming that $I_1$ holds, so  
$\{N^d_{-(m_2+1)j_1,0} \geq (\kappa -\eps)(m_2+1)j_1 \}$ holds.
If we take $n=(m_2+1)j_1$ then ${\Theta}\geq (\kappa -\eps)n$,
and therefore
\begin{align*}
|X(0)-Y(0)| = e^{-\rho {\Theta}} \leq e^{-\rho (\kappa -\eps)n}\leq e^{-n}.
\end{align*}
This proves that $C_1$ holds.

\emph{Step 5.} 
It remains to bound the probability that $A_k\cap E^k$ fails.
We have proved the following three inequalities:
\begin{align}\label{quad-n5.3}
S^d_{X,k} < S^d_{Y,k} < S^u_{X,k} < S^u_{Y,k},
\end{align}
illustrated by (a), (b) and (c) (blue, red and green) in Fig. \ref{fig11}.
To complete one ``cycle'' in the definition of $A_k$ one has to prove that
\begin{align*}
S^u_{X,k} < S^u_{Y,k}
< S^d_{Y,k+1} < S^d_{X,k+1} ,
\end{align*}
illustrated (partly) by (c) and (d) (green and orange) in Fig. \ref{fig11}.
The event $A_k$ also contains another ``cycle,'' with the roles of $X$ and $Y$ interchanged.
Each of the four sets of inequalities (two per ``cycle'') are proved similarly 
to the proof of \eqref{quad-n5.3}.
 This explains the factor of $1/4$ in the first line of the estimate below.

An event $A_k\cap C_2^k$ fails only if $ I_1 \cap I_2$ fails or $F_1(T)\cap F_2(T)$ fails with $T=S^d_{X,k}$.
Let $p=(1-p_1)/12$.
By Proposition \ref{s25.1} and Lemmas \ref{s16.1} and \ref{n3.6}, 
\begin{align*}
\tfrac14 \P^h_{\nu(e^{-n})}&
\left(
\bigcup_{k=1,\dots K-1,\  k \text{  odd}} A_k^c
\right)
\leq 
\P^h_{\nu(e^{-n})} \left (I_1^c)\right)
+\P^h_{\nu(e^{-n})}\left (I_2^c)\right)
\\
&\qquad + \sum_{m\geq m_1} \P^h_{\nu(e^{-n})}
\left(F_1(T_{e^{mj_1}})\cap I_1\cap I_2 \right)
+ \sum_{m\geq m_1} \P^h_{\nu(e^{-n})}
\left(F_2(T_{e^{mj_1}})\cap I_1\cap I_2 \right)
\\
&\leq
p+p \\
&\quad +\sum_{m=1}^{\infty} (\kappa + \eps) m j_1 
\left(
c_1 e^{-2mj_1\eps/\alpha}
+c_3
 e^{-m j_1\eps} 
\exp\left( -c_4 e^{-2mj_1 \eps} \right)
\right)\\
&\leq 3p=\tfrac14(1-p_1).
\end{align*}
The last inequality holds provided we choose $j_1$ sufficiently large.
The factor $(\kappa + \eps) m j_1$ on the second to last line is there to account for multiple $k$'s corresponding to each $m$; being on the event $I_1$ 
controls  the number of such $k$. 
As long as $n$ is large enough that $e^{-n}\le c_1/4$, an
 application of Lemma \ref{C2-lemma} establishes  $C_2$, and
the proof  is complete.

Note that the choice of $p_1$ affects $j_1$ and this in turn affects the value of the constant $c_1 =  e^{-\rho(\kappa + \eps)m_1 j_1} $ in the definition of $C_3$.
\end{proof}

We now obtain a suitable version of Corollary \ref{harnack-c1}(i) for measures 
on $D$.

\begin{lemma}\label{quad-L512} Fix $n\ge 1$. There exists 
a positive integer $j_1>n$  such
that for each Borel subset $A$ of $H_{e^{-n}}$, each $j\ge j_1$, and each
$z\in H_{e^{-j}}$ we have
$$|\P^h_{z}(X(T_{e^{-n}})\in A)-\nu(e^{-n})(A)|\le 2q^{j-n-1},$$
where $q$ is given in the statement of Corollary \ref{harnack-c1}.
\end{lemma}

\begin{proof}
By Corollary \ref{harnack-c1}(i) and scaling, we have
$$|\wt \P^{\wt h}_{\F(z)}(\wt X(T_{e^{-n}})\in A_1)-\wt \P^{\wt h}_{\wt\nu(e^{-j})}(\wt X(T_{e^{-n}})
\in A_1)|
\le q^{j-n-1}/(1-q^{j-n-1}),$$
 where $A_1=\F(A)$ is a Borel 
subset of $\wt H_{e^{-n}}$ and $\F(z)\in \wt H_{e^{-j}}$.
The quantities above are not affected by time changes, so by scaling and the
fact that $\wt\nu(e^{-n})$ is an invariant probability measure, we obtain
\begin{equation}\label{harnack-e512.1}
|\wt\P^{\wt h}_{\F(z)}(\F(X(T_{e^{-n}}))\in \F(A))-\wt\nu(e^{-n})(\F(A))|
\le 2q^{j-n-1}
\end{equation}
provided $j\ge j_1$ and $j_1$ is chosen so that $q^{j_1-n-1}<1/2$,
where $A$ is a Borel subset of
$H_{e^{-n}}$ and $z\in H_{e^{-j}}$.

From the definition of $h$-path transform, we have
\begin{align*}
\wt\P^{\wt h}_{\F(z)}(\F(X&(T_{e^{-n}}))\in \F(A))
=\wt\E_{\F(z)}[\wt h(X(T_{e^{-n}})); \F(X(T_{e^{-n}}))\in \F(A)]/\wt h(z)\\
&=\frac{e^{-n}}{e^{-j}}\wt \P_{\F(z)}(\F(X(T_{e^{-n}}))\in \F(A))
=\frac{e^{-n}}{e^{-j}}\P_{z}(X(T_{e^{-n}})\in A)\\
&=\P^h_z(X(T_{e^{-n}})\in A).
\end{align*}
Substituting in \eqref{harnack-e512.1} yields our result.
\end{proof}

\section{Non-uniqueness}\label{sect-nonuni}

First we argue that to prove non-uniqueness it is enough to look at excursions
of a solution to \eqref{def-skor}.

Let $B$ be 2-dimensional Brownian motion and let $X$ be a solution to 
\eqref{def-skor}. Next kill (in the Markov sense) the process $X$ at time $T_1$.
Let 
$$U_0=\sup\{t\le T_1: X_t=0\}.$$
Following \cite{MSW} we see that $X(U_0+t)$ is a strong Markov process whose 
law is given by $\P^{h_0}_\cdot$, where $h_0(z)=\P_z(T_1<T_0)$. (Strictly speaking,
we should write $\P_z(\zeta<T_0)$, where $\zeta$ is the 
lifetime of $X$ killed on hitting $H_1$.)

The function $h$ defined in \eqref{def-h1} is equal to 1 on $H_1$, 0 at 0,
and as we observed, $h(X_t)$ is a martingale under $\P_\cdot$. Therefore by optional stopping,
$h(z)=\P_z(T_1<T_0)=h_0(z)$. 

Let
$U_b=\inf\{t\ge U_0: X_t\in H_b\}$.

\begin{lemma}\label{nonuni-L1}
(i) Let $n\ge 1$. Then
$$\lim_{j\to \infty} \sup_{z\in H_{e^{-j}}}\sup_{A}
 |\P_0(X(U_{e^{-n}})\in A\mid X(U_{e^{-j}}) =z)-\nu(e^{-n})(A)|=0,$$
where the supremum over $A$ is the supremum over  Borel subsets $A\subset H_{e^{-n}}$.\\
(ii) For $n\ge 1$,
$$\P_0(X(U_{e^{-n}})\in A)=\nu(e^{-n})(A)$$
for all Borel subsets $A\subset H_{e^{-n}}$.
\end{lemma}

\begin{proof}
(i) Let $A\subset H_{e^{-n}}$, let $\eps>0$, and let $j\ge j_1$ so that
$q^{j-n-1}\le \eps/2$, where $q$ and $j_1$ are given by the statement
of Lemma \ref{quad-L512}.
By the conclusion of Lemma \ref{quad-L512}, if $z\in H_{e^{-j}}$,
$$|\P^h_z(X(T_{e^{-n}})\in A)-\nu(e^{-n})(A)|<\eps.$$
Our result follows because the distributions of
$\{X(U_0+t), t\geq 0\}$ under $\P_\cdot$ and $\{X_t,t\ge 0\}$ under $\P^h_\cdot$ are identical and  $\eps>0$ is
arbitrary.

(ii) Let $j>n$. 
Our result follows because $U_{e^{-j}} < \infty$, a.s, so $j$ can be taken arbitrary large in (i), and, 
therefore $\eps>0$ can be taken arbitrarily small.
\end{proof}

We now apply Theorem \ref{o28.1} with $p_1=1/2$ to construct a sequence of processes $\{Y^{(n)}\}$. 
For
each $n\ge n_1$, let $Y^{(n)}_t=X_t$ for $t< U_{e^{-n}}$. At time
$U_{e^{-n}}$  we introduce a jump. 
Let $\Theta_n$  be 
as in the proof of Theorem \ref{o28.1}, but we write $\Theta_n$ instead of $\Theta$ to emphasize the dependence on $n$. 
We define
\begin{align}\label{j25.1}
Y^{(n)}(U_{e^{-n}})=
\begin{cases}
X(U_{e^{-n}})+(0, e^{-\psi \Theta_n/\kappa}) &
\text{  if  } e^{-\psi \Theta_n/\kappa} < e^{-n},\\
X(U_{e^{-n}}) & \text{  otherwise}.
\end{cases}
\end{align}
We let
$Y^{(n)}$ be the solution to 
$$Y^{(n)}(U_{e^{-n}}+t)=Y^{(n)}(U_{e^{-n}})+B(U_{e^{-n}}+t)-B(U_{e^{-n}})
+\int_{U_{e^{-n}}}^{U_{e^{-n}}+t} {\bf v}(Y^{(n)}_s) \, dL^{Y^{(n)}}_s,$$
for $0\leq t \leq U_*:=\inf\{s: Y^{(n)} _s=0\}$,
where $B$ is the same as in \eqref{def-skor} and Theorem \ref{o28.1}.
As noted in Remark \ref{R-1-dim}, the solution to this equation is pathwise
unique up to time $U_*$.

We now use Proposition \ref{prelim-Lcompletely} to find a pathwise solution to 
\eqref{def-skor} for $t>U_*$:
$$Y^{(n)}(U_{*}+t)=Y^{(n)}(U_{*})+B(U_{*}+t)-B(U_*)
+\int_{U_{*}}^{U_{*}+t} {\bf v}(Y^{(n)}_s) \, dL^{Y^{(n)}}_s$$
for $t>0$.
We do not know whether or not the solution $Y^{(n)}$ is pathwise unique 
 for times after
$U_*$, but that will not matter, as it is the behavior before time $U_*$ that
will be used to show that the solution to \eqref{def-skor} is not
unique.

In view of \eqref{j25.1},
 Theorem \ref{o28.1} shows that  we have processes $Y^{(n)}$ such that 
$$\P_0\Big(\sup_{s\le U_*}
|X_s-Y^{(n)}_s|\ge c_1\Big)\geq 1/2.$$
The size of the jump of $Y^{(n)}$ at the time $U_{e^{-n}}$ is less than or equal to $e^{-n}$, a.s.

We next show tightness of the sequence $(X,B,Y^{(n)})$. 
Choose $t_0$ large so that $\P(T_1>t_0)\le 1/4$.

\begin{lemma}\label{nonuni-L2}
The triple $(X,B,Y^{(n)})$ is tight with respect to the topology of 
$D[0,t_0]$.
\end{lemma}

Recall that each of $X,B,Y^{n}$ is bivariate.

\begin{proof}
Since $X$ and $B$ are continuous processes, it suffices to show that 
$\{Y^{(n)}\}$ is tight. We use a criterion of Aldous, namely, 
tightness holds if
whenever $\tau_n$ are stopping times and $\delta_n$ are positive reals
tending to 0, then 
\begin{equation}\label{aldous}
|Y^{(n)}(\tau_n+\delta_n)-Y^{(n)}(\tau_n)|\to 0
\end{equation}
in probability (see \cite[page 264]{Ba-stoch}). 
An examination of the proof of the Aldous criterion 
shows that we need \eqref{aldous} 
to hold for stopping times $\tau_n$ such that for each $n$, the time $\tau_n$ is
a stopping time with respect to the filtration generated by $Y^{(n)}$. However we will establish 
\eqref{aldous} for arbitrary random times $\tau_n$, not necessarily stopping times, so the choice of a filtration is a moot point.

The size of the jump of
$Y^{(n)}$ at time $U_{e^{-n}}$ is less than or equal to $e^{-n}$, a.s., which tends to 0
  as $n\to \infty$.
By Proposition \ref{prelim-Lcompletely}(ii) we have that if $U_{e^{-n}}\in [\tau_n,\tau_n+\delta_n]$, then
\begin{align*}
|Y^{(n)}(\tau_n+\delta_n)-Y^{(n)}&(U_{e^{-n}})|
+|Y^{(n)}((U_{e^{-n}})-)-Y^{(n)}(\tau_n)|\\
&\le 2\chi_1 \sup_{0\le s<t\le s+\delta_n\le t_0} |B_t-B_s|,
\end{align*}
which also goes to 0 as $\delta_n\to 0$.
If $U_{e^{-n}}$ is not in that interval, we have the same bound for
$$|Y^{(n)}(\tau_n+\delta_n)-Y^{(n)}(\tau_n)|$$
by the same proposition.
Therefore the Aldous criterion is satisfied.
\end{proof}

We use the Skorokhod representation theorem  (cf.\ the proof of 
Proposition \ref{invar-p1}) to find a probability
space $(\Omega', {\calF}',\P')$ and random processes $X'_n, B'_n, Y'_n$
such that the law of $(X'_n,B'_n,Y'_n)$ under $\P'$
is equal to the law of $(X,B, Y^{(n)})$ under $\P_0$ and 
$(X'_n,B'_n,Y'_n)$ converges a.s.\ to a triple $(X'_\infty,B'_\infty, Y'_\infty)$ with respect to the topology of $D[0,t_0]$. Clearly 
$B'_\infty$ is a Brownian motion. 
The processes $X$ and $ Y^{(n)}$ have oscillations bounded by a constant times those of Brownian motion $B$,  by Proposition \ref{prelim-Lcompletely}(ii). This implies that the same bounds for the oscillations apply to the limit
$(X'_\infty, B'_\infty,Y'_\infty)$. Thus these processes are continuous, and it is well known (see, e.g. \cite[page 263]
{Ba-stoch}) that therefore the convergence of $(X'_n,B'_n,Y'_n)$ is uniform, a.s.

To verify that $(X'_\infty,B'_\infty)$ and $(Y'_\infty,B'_\infty)$
are both solutions to \eqref{def-skor}, it is more convenient to 
look at the formulation given by \eqref{def-HR}.

\begin{lemma}\label{nonuni-L3}
$(X'_\infty,B'_\infty)$ and $(Y'_\infty,B'_\infty)$ are both solutions to
\eqref{def-HR}. Moreover
$$\P\Big(\sup_{s\ge 0} |X_s-Y_s|>0\Big)\geq 1/4.$$
\end{lemma}

\begin{proof}
Only a few steps need comment. The determinant of $R$ is equal to 
$1-a_1a_2>0$, hence $R$ is invertible. Therefore $(M^{Y'_n})^i$ converges uniformly 
on compacts to $(M^{Y'_\infty})^i$ for $i=1,2$.

The other step is to argue that $(M^{Y'_\infty})^1$ increases only when
$Y'_\infty\in \Gamma_d$ and similarly for $(M^{Y'_\infty})^2$.
If $Y'_\infty(t)\notin \Gamma_d$, then 
$Y'_\infty(s)\notin \Gamma_d$ for $s$ in an open interval about $t$, and hence
for $n$ large, $Y'_n(s)\notin \Gamma_d$ for $s$ in this interval.
This means that $(M^{Y'_n})^1(s)$
is constant for $s$ in an open interval containing $t$ if $n$ is sufficiently
large. But then the same is true for $(M^{Y'_\infty})^1(s)$.
\end{proof}

The last step in proving that pathwise uniqueness fails almost surely is
to show that the probability of non-uniqueness is 1 rather than only saying it 
is greater than $1/4$.

Let ${\calC}_2[0,t_0]$ be the set of functions from $[0,t_0]$ into $\R^2$ such that
each component is continuous.
Let $x_0=0$.
Let ${\calU}$ be the set of functions $f$ in ${\calC}_2[0,\infty)$ for which 
\eqref{prelim-e100.1} has a unique solution for each $t_0>0$.

\begin{theorem}\label{nonuni-T1}
With $\P_0$ probability one,  $$B_t(\omega)\notin  {\calU}.$$
\end{theorem}

\begin{proof}
Let $T^1=T_1$, $U^0=0$, $$U^k=\inf\{t>T^k: X(t)=0\}, \qquad k\ge 1,$$
and
$$T^{k+1}=\inf\{t>U^k: X_t\in H_1\}, \qquad k\ge 1.$$
By the strong Markov property, the laws of  $\{(X_t,B_t): t\in [U^k, T^{k+1}]\}$
are i.i.d. Therefore the probability that $B\in {\calC}_2[0,\infty)\setminus {\calU}$
is less than $(3/4)^k$ for each $k$, hence is 0.
\end{proof}

We now turn to showing there exists no strong solution to \eqref{def-skor}
and equivalently, to \eqref{def-HR}. We first need the following lemma.

Recall that  $(X,B)$ solving  \eqref{def-HR} is called a weak solution if $B$ is adapted to the
filtration generated by $\{X_t\}$.
 
\begin{lemma}\label{nonuni-L10}
Suppose $B_t$ is a 2--dimensional Brownian motion and $(X,B)$ 
satisfies \eqref{def-HR}. Then $(X,B)$ is a weak solution of \eqref{def-HR}.
\end{lemma}

\begin{proof}
By Lemma 2.1 of \cite{TW}
$$\int_0^t 1_{\partial D}(X_s)\, ds=0, \qquad\mbox{a.s.}$$
Therefore 
$$\E\Big[\int_0^t 1_{\partial D}(X_s)\, dB_s^i\Big]^2=0, \qquad i=1,2,$$
and so
\begin{equation}\label{eq-weak-sol1}
\int_0^t 1_{\partial D}(X_s)\, dB_s^i=0, \qquad\mbox{a.s.},
\end{equation}
for $i=1,2$. Since $M_t^i$ increases only when $X_t\in \partial D$, $i=1,2$,
then
\begin{equation}\label{eq-weak-sol2}
\int_0^t 1_{D^o}(X_s)\, dX_s^i=\int_0^t 1_{D^o}(X_s)\, dB_s^i, \qquad i=1,2,
\end{equation}
where $D^o$ is the interior of $D$. 
Combining \eqref{eq-weak-sol1} and \eqref{eq-weak-sol2}
$$B_t^i=\int_0^t 1_{D^o}(X_s)\, dB_s^i+\int_0^t 1_{\partial D}(X_s) \, dB^i_s
=\int_0^t 1_{D^o}(X_s)\, dX_s^i,$$
which implies that $B$ is adapted to the filtration generated by $\{X_t\}$.
\end{proof}      

\begin{theorem}\label{nonuni-T2}
There does not exist a strong solution to \eqref{def-HR}.
\end{theorem}

One could modify existing proofs of the analogue of this theorem 
for ordinary stochastic differential equations, but it is just as easy to use Girsanov's original proof, slightly modified.

\begin{proof}
Let $B_t$ be a 2-dimensional Brownian motion and suppose $(X,B)$ is a strong
solution. We are not supposing here that this $X$ is the same as  any of the
processes we have worked with up to now. We will prove that the solution
to \eqref{def-HR} is then pathwise unique, which contradicts Theorem \ref{nonuni-T1}.

Let $(Y,B)$ be any other solution to \eqref{def-HR}. By Lemma \ref{nonuni-L10}
$(Y,B)$ is a weak solution. By \cite{TW} the law of $X$ is unique, hence
the law of $X$ equals the law of $Y$. 
By Lemma \ref{nonuni-L10} strong solutions are also weak solutions, hence the
law of $(X,B)$ is equal to the law of $(Y,B)$. 

We know that with positive probability there is not uniqueness to
\eqref{prelim-e100.1} with $f(t)$ replaced by $B_t(\omega)$. Also $X$ and $Y$
are continuous processes. Therefore there
 exists a positive rational $r$ such that
$\P( X_r\ne Y_r)>0$. Since $X_r$ is adapted to the filtration generated
by $\{B_s: s\le r\}$, there is a Borel measurable map $\varphi$  from 
${\calC_2[0,r]}$ to $D$ such that $X_r=\varphi(B)$ a.s. Because the 
laws of $(X,B)$ and $(Y,B)$ are equal, we must have that $Y_r$ also equals 
$\varphi(B)$ a.s. But then $X_r=Y_r$ a.s., a contradiction.
\end{proof}

\section{Index of notation}\label{sect-notation}

We present an index of notation that is used repeatedly.
\medskip

\noindent Functions:
 $\wt h$ \eqref{def-h2};\quad   $h$ \eqref{def-h1};    \quad $\F$ \eqref{def-f}
\medskip

\noindent Matrices:
 $R$ \eqref{R-def}

\noindent Measures and probabilities:
$H^x$ \ Section \ref{sect-excursions}; \quad   $\P^h$ \eqref{def-Ph};
\quad   $\wt \P^{\wt h}$ \eqref{def-hpath}
\medskip

\noindent Parameters:
$a_1,a_2$ \eqref{def-ai};\quad $\alpha$ \eqref{def-alpha};
\quad  $\beta$ \eqref{def-beta};\quad $\kappa$ \eqref{def-kappa}; \quad  $\theta_1, \theta_2$ \ Section \ref{sect-intro};  \quad $\chi_1$ Proposition \ref{prelim-Lcompletely};
\quad $\psi$ \eqref{def-psi}
\medskip

\noindent Processes:
$L$ \eqref{def-skor}; \quad  $X$ \eqref{def-skor};  \quad $\wt X$ \ Proposition
\ref{conf-p1}; \quad $\wh X$ \ Proposition \ref{conf-p2}
\medskip

\noindent Random variables: 
$\wt N^d_{k_1,k_2}, \wt N^u_{k_1,k_2}$ \eqref{def-wtN}; \quad $S_k^d, S_k^u$ (with and without tildes and hats) \eqref{def-S}; \quad $S_k^{d-}; S_k^{u-}$ (with and without tildes and hats) \eqref{def-Sminus}; \quad $T_b$ (with and without
tildes and hats) \eqref{def-Tb}
\medskip

\noindent Sets:
$ D$ \ Section \ref{sect-intro}; \quad $\wt D, \wh D$ \ Section \ref{sect-conformal};
\quad  $H_b, \wt H_b, \wh H_b$ \eqref{def-H}; \quad $\Gamma_d, \Gamma_d$ \eqref{def-gamma}
\medskip

\noindent Vectors:
${\bf n},  {\bf v} $ \ Section \ref{sect-intro}
\medskip

\bibliographystyle{alpha}

\end{document}